\documentclass[final]{siamart220329}
\usepackage{amsmath,amsfonts,amsopn,cite,mathrsfs,amssymb} 
\usepackage{mathtools}
\usepackage{graphicx}
\usepackage{epstopdf}
\usepackage{algorithmic}
\usepackage[a4paper,left=3cm,right=3cm,top=3.5cm,bottom=3.5cm]{geometry}
\ifpdf
  \DeclareGraphicsExtensions{.eps,.pdf,.png,.jpg}
\else
  \DeclareGraphicsExtensions{.eps}
\fi

\newcommand{\field}[1]{\mathbb{#1}}
\newcommand{\R}{\field{R}}        
\newcommand{\N}{\field{N}}        
\newcommand{\Z}{\field{Z}}        
\newcommand{\C}{\field{C}}        
\newcommand{\ft}{Fourier transform}

\def\cPp{\mathcal{P}\!_p}
\def\cPh{\mathcal{P}\!_{h^{-1}}}
\def\cO{\mathcal{ O}}
\def\inv{^{-1}}
\DeclareMathOperator*{\supp}{supp}

\DeclareMathOperator{\DFTa}{\mathcal{F}}

\def\DFT{\DFTa}

\DeclareMathOperator*{\sinc}{sinc}
\DeclareMathOperator*{\sign}{sign}

\newsiamremark{remark}{Remark}
\newsiamremark{hypothesis}{Hypothesis}
\crefname{hypothesis}{Hypothesis}{Hypotheses}
\newsiamthm{claim}{Claim}

\headers{Discrete Fourier transform vs.~Fourier transform}{M. Ehler, K. Gr\"ochenig, and A. Klotz}

\title{Quantitative estimates: How well does the discrete Fourier transform approximate the Fourier transform on $\R$?}

\author{Martin Ehler\thanks{Faculty of Mathematics, University of Vienna, Oskar-Morgenstern-Platz 1, A-1090 Vienna, Austria
  (\email{martin.ehler@univie.ac.at}, \url{https://homepage.univie.ac.at/martin.ehler/}).}
\and Karlheinz Gr\"ochenig\thanks{Faculty of Mathematics, University of Vienna, Oskar-Morgenstern-Platz 1, A-1090 Vienna, Austria
  (\email{karlheinz.groechenig@univie.ac.at}, \url{https://homepage.univie.ac.at/karlheinz.groechenig/}).}
\and Andreas Klotz\thanks{Acoustics Research Institute, Austrian Academy of Sciences,
Dominikanerbastei 16, A-1010 Vienna, Austria (\email{andreas.klotz@oeaw.ac.at}).}}
\usepackage{amsopn}

\begin{document}
\maketitle

\begin{abstract}
In order  to compute  the Fourier transform
of a function $f$ on the real line numerically, one samples  $f$ on a
grid and then takes the discrete Fourier transform. We derive exact
error estimates for this procedure in terms of the decay and
smoothness of $f$. The analysis provides an asymptotically optimal recipe of how to relate
the number of samples, the sampling interval, and the grid size. 
\end{abstract}



\section{Introduction}
The fast Fourier transform (FFT) is widely used in the applied
sciences for the  numerical approximation of  the Fourier transform 
\begin{equation}\label{eq:101}
\hat{f}(\xi)=\int_{-\infty}^\infty f(x)e^{-2\pi \mathrm{i}x\xi}\mathrm{d}x
\end{equation}
from sampled values of $f$. Despite the overwhelming  success of this
approximation,  there are surprisingly few rigorous investigations
and, in our view, 
still substantial theoretical
gaps when it comes to error estimates. As the FFT is a fast algorithm
that computes the discrete Fourier transform, we are actually asking:
\emph{how well does the discrete Fourier transform approximate the
  Fourier transform \cite{Epstein}?} 

\subsection{Questions}
Engineers and numerical practitioners compute the Fourier transform of
a function $f$ on the real line as follows: 
\begin{itemize}
\item[(i)] Sample $f$ on an interval $(-\frac{p}{2},\frac{p}{2}]$ of length $p$  at $n$ equispaced points
$hj$, for $j\in \Z $, $-\frac{n}{2} < j\leq \frac{n}{2}$, with step
size  $h$ and $p=hn$. 

\item[(ii)]  Approximate the Fourier transform of $f$ at
    $\tfrac{k}{p}$ by
\begin{equation}
  \label{eq:c1}
\hat{f}\big( \tfrac{k}{p} \big)\approx h\sum_{-\frac{n}{2}<j\leq \frac{n}{2}} f(hj)  e^{-2\pi
    \mathrm{i}\frac{k j}{n}}  \,,\qquad k\in\Z,\quad -\frac{n}{2}< k \leq \frac{n}{2} \, . 
\end{equation}
\end{itemize}
The right-hand side is the discrete Fourier transform of the
  sampled vector  $\big(f(hj )\big)_{j}$ of size $n$.  By
  choosing $n$ to be a power of $2$ or by zero-padding, the discrete Fourier
  transform can be computed with the FFT algorithm in 
  $\cO (n \log n) $ operations.  Since the \ft\ is a fundamental
  computational step   in a huge number of applications  in signal
  processing, acoustics, medical imaging, the  numerical solution of
  partial differential equations, in quantum mechanics, 
  etc., the availability of a fast algorithm is central  in
  scientific computing. Therefore much effort has gone into a speed up and
  efficient implementations of this algorithm~\cite[Chapter 12]{numerical recipes}. The FFT is
  hailed as  ``the most important numerical algorithm of our
  life-time''~\cite{Strang}, is included in the ten most important numerical
  algorithms of the 20th century~\cite{20thcentury}, and remains important in the age of data science and artificial intelligence.
  
Grounded on  the overwhelming  success of the FFT in 
approximating the \ft , the approximation \eqref{eq:c1}  is usually
taken for granted. The FFT is used as a black box, it  works superbly, and
it always seems to work. \\
\indent
Qualitatively, the approximation~\eqref{eq:c1} is easy to
understand. It is motivated by discretizing the integral and then
truncating the infinite series as follows: 
\begin{equation}\label{eq:1approx}
  \hat f (\xi) =
  \int_{-\infty}^\infty f(x)e^{-2\pi \mathrm{i}\xi x}\mathrm{d}x  \approx h\sum_{j\in\Z} f(hj) e^{-2\pi\mathrm{i}\xi hj }
  \approx h\sum_{-\frac{n}{2}<j\leq \frac{n}{2}} f(hj)
  e^{-2\pi\mathrm{i}\xi hj }\, . 
\end{equation}
For $\xi = \frac{k}{hn} = \frac{k}{p}$ one then obtains \eqref{eq:c1}. \\
\indent  Although this approximation is immensely successful,   one
must ask what is the precise relation between the computed 
values $h\sum_{-\frac{n}{2}<j\leq \frac{n}{2}}  f(hj)  e^{-2\pi
    \mathrm{i}\frac{k j}{n}}$
and the samples $\hat{f}(\frac{k}{p})$ of the actual 
Fourier transform.

Ideally one can provide reasonable error estimates for
this approximation, and we formulate two questions:
\begin{itemize}
 \item[(Q1)] The numerical procedure depends on the number of samples
   $n$,  the step size $h$, and  the length $p$ of the sampled interval
where $p=hn$. How should $h,n$, and $p$ be  chosen  dependent on the function
class of $f$?

\item[(Q2)]
What can be said about error rates and the asymptotic decay of the deviation in
\eqref{eq:c1} when $n \rightarrow \infty $, $p\rightarrow \infty$, and $h\to 0$? 
\end{itemize}

We will answer these questions for several classes of functions that
are described by (i) their decay property and (ii) by their
smoothness or by the decay of their Fourier transform. The decay of
$f$ determines the truncation error in~\eqref{eq:1approx}, and the decay of
$\hat{f}$ determines the discretization error in \eqref{eq:1approx}.

\subsection{State-of-the-art} 

The case of functions with compact support is related to the 
case of approximating Fourier coefficients of a periodic function and is fully
understood. Roughly, if $f$ has compact support and is $m$-times
differentiable, then the pointwise  error scales like $n^{-m}$ in
terms of length of the FFT, or like $h^m$ in terms of the grid
size. This is made precise by Epstein in \cite{Epstein}, whose title we
borrowed, for the
approximation of Fourier coefficients of periodic functions, and by 
Briggs and Henson in their ``Owner's Manual for the Discrete Fourier
 Transform''~\cite[Section 6]{Briggs} for the Fourier transform of
 compactly supported functions on $\R $. For $\supp f\subseteq
 [-\frac{p}{2},\frac{p}{2}]$ and  $h=\frac{p}{n}$,  these results are optimal and
 come with explicit constants. \\
\indent 
 The article~\cite{Becker96}
  computes an exact formula for the relative pointwise error for the
  \ft\ of a compactly supported $B$-spline of order $k$ (called ``canonical$-k$
functions'') and shows that in the limit $n\to \infty $ the relative error
is the same for $k+1$-times differentiable functions with compact
support, however,  no error estimates are given. \\
\indent 
To the best of our knowledge, there are surprisingly few
 investigations and even fewer quantitative error estimates in the literature that go beyond compactly supported functions.
For functions with unbounded support 
the
important cases of exponential decay and analyticity of $f$ are
covered by  the
work of  Briggs
and Henson~\cite[Chapter 6]{Briggs}
and of  Stenger \cite[Chapter 3.3]{StengerBook}.  \\
\indent 
The case of general input in \cite[Thm.~6.6]{Briggs} assumes that  $f$
satisfies $\sup_{x \in \mathbb{R}} |f(x)| e^{r|x|} < \infty$ and is
$m$-times differentiable satisfying some stringent boundary conditions
at $\pm p/2$. 
 Then for some  constants $c_1,c_2$ the
pointwise  error is 
\begin{equation}\label{eq:bound from briggs}
\Big| \hat{f}\big( \tfrac{k}{p} \big) - h \sum_{-\frac{n}{2}<j\leq \frac{n}{2}}f(hj)
  e^{-2\pi i \frac{kj}{n}} \Big| \leq c_1 e^{-r p/2} e^{-r p/n} + c_2
h^{m}, \qquad -\frac{n}{2}< k \leq \frac{n}{2} \, .
\end{equation}
The decay of order $h^{m}$ is as expected, and the implicit assumption
of bounded variation \cite{Briggs} even leads to $h^{m+1}$. The problem, however, is in the
boundary conditions. In the generic case, when the values of $f$ at
the endpoints of the sampled interval differ, i.e., $f(-p/2) \neq
f(p/2)$, then the predicted error is only of the order $h$. For
instance, for the exponential function $f(x) = e^{-2\pi |x|}$
\cite{Briggs} obtains the correct error estimate with $h^2$, but for
the shifted function $e^{-2\pi |x-1|}$ the predicted error is only
$h^1$, which does not match the numerical observations. 
  This phenomenon seems to be a limitation of the proof
technique in \cite{Briggs}. It is one of our contributions to obtain
the correct error estimate
without any boundary conditions. For more discussion see our Section~\ref{sec:mw}.

In the context of numerical sinc methods 
Stenger \cite[Chapter 3.3]{StengerBook}  uses  the
cardinal series to  approximate an exponentially decaying function
$f$ whose \ft\ also decays exponentially (to be precise, he uses a
slightly stronger condition formulated with a Hardy  space) 
and shows that the pointwise error satisfies
$$
\Big| \hat{f}\big( \tfrac{k}{p} \big) - h \sum_{-\frac{n}{2} < j \leq \frac{n}{2}} f(hj)
  e^{-2\pi i \frac{kj}{n}} \Big| \leq C e^{-s n^{1/2}} \, ,
$$
with suitable constants $C,s>0$. In addition,  \cite[Chapter
3.3]{StengerBook} contains error estimates for the $L^p$-error for
interpolated versions of the \ft.

Since the right-hand side of \eqref{eq:c1}  approximates the
integral $\int _{-p/2}^{p/2} f(x) e^{-2\pi \mathrm{i} x k/p} \mathrm{d} x$ by means of a
Riemann sum, one should mention the
recent error analyses~\cite{Eggert,Fornberg21,Tadmor,Trefethen14,Trefethen22}  for
the trapezoidal rule as relevant in this context although they do not 
mention the Fourier transform. These error estimates assume
exponential decay and analyticity of $f$. By contrast, the recent
results for the trapezoidal rule in \cite[Prop. 4.2, Thm 4.5]{Kaza} do
not require analyticity, but 
use Gaussian decay  and finite Sobolev smoothness of order $m$. For
$n$ samples of $f$ and the choice $
p=2\sqrt{\frac{2}{1-\epsilon} m \log n}
$ the error is 
of the order $n^{-m} (\log
n)^{m/2+1/4}$. \\
\indent In \cite{Denich} the trapezoidal rule is  used for error estimates of
  approximations of the Fourier transform of analytic functions via
  doubly exponential transforms. These  are rather different
from the usual approximation~\eqref{eq:c1}, but yield almost
exponential convergence.

An interesting variation was studied by Auslander and
Gr\"unbaum~\cite{Auslander89}. They consider  functions  $f\in
L^2(\R)$ and approximate averaged values of $\hat{f}$ from local averages
of $f$ via the discrete Fourier transform. For Gaussian smoothing they
derive an explicit  formula for the sharp  constant  for the resulting error
estimate and then plot the results of numerical simulations for FFTs
up to $n=100$ samples. Asymptotics for large $n$ is not treated.

\section{Our contributions} 
In this paper we derive estimates for the
error between the Fourier transform on the real line and the discrete Fourier
transform for a significantly broader class of functions than previously considered in the literature. 

  Specifically, we study  functions
  that decay only polynomially in time and in frequency.   
Our results confirm in a quantitative fashion  that, even under mild assumptions on the decay and
smoothness (decay in frequency), the standard approximation 
procedure with the FFT works well and is successful.
Our analysis also 
provides the new insight how the optimal spacing should depend on the
decay of $f$ and $\hat{f}$. As an answer to Question (Q1), we identify an optimal relation between the step size $h$, the number of samples $n$, and the length $p$ of the sampling interval. To answer Question (Q2), we derive precise error
estimates with explicit constants for function classes that are
significantly larger than exponentially decaying or analytic
functions. The optimality of the  error estimates  is 
confirmed by numerical simulations with  the FFT that  yield the precise
asymptotics predicted by  theoretical results. We believe that these error estimates are best possible, and 
 since all constants are explicit, these   estimates 
  cover  not only the asymptotic regime,  but also hold for small
  $n$.

We proceed with the detailed exposition of our results. 
\subsection{Error measures}
We first introduce the appropriate notation.
For $n\in\N$, set
\begin{equation*}
  [n]:=
  \{j\in\Z : -\tfrac{n}{2}<j\leq \tfrac{n}{2}\}\,.
\end{equation*}
The discrete Fourier transform $\DFT:\C^n\rightarrow\C^n$ of  $y=(y_j)_{j\in[n]}\in \C ^n$ is   
\begin{equation}\label{eq:DFT}
\DFT\! y = \bigg(\frac{1}{\sqrt{n}}\sum_{j\in[n]} y_j e^{-2\pi
  \mathrm{i}\frac{k j}{n}}\bigg)_{k\in[n]} \, .
\end{equation}
Define  the  scaled sampling of a function $f$ with step size $h$ and length $n$ as
\begin{equation*}
f_{h,n}:= \Big(\sqrt{h}f(hj)\Big) _{j\in[n]} \in \C^n\,.
\end{equation*}
Throughout the text we use the relation  $p=hn $ between the interval length $p$, the step size $h$, and the number of samples $n$.

We will deal mainly with the $\ell ^2$-error averaged over
  the sampled interval. This  approximation error is defined as  
\begin{equation}\label{eq:error original unique}
E^{[n]}_{h}(f):=\biggl(\frac{1}{p}\sum_{k\in [n]} \Big| \hat{f}(\tfrac{k}{p}) - h\sum_{j\in[n]} f(hj) e^{-2\pi\mathrm{i}\frac{kj}{n}}\Big|^2\bigg)^{1/2} = \big\| \hat{f}_{\frac{1}{p},n} - \DFT\! f_{h,n}\big\| \,.
\end{equation}
 The use of the $\ell^2$-error is natural in our context, since the discrete Fourier transform is unitary on $\C^n$ with respect to the $\ell^2$-norm. 
 The normalization of $f_{h,n}$ and the factor $\frac{1}{p}$ 
are motivated by 
 the remarkable time-frequency symmetry of the error
$$E^{[n]}_h(f) =
E^{[n]}_{p^{-1}}(\bar{\hat{f}}) \, ,
$$
see Lemma \ref{lemma:symmetry}. This symmetry aligns well with the
mathematical framework of time-frequency analysis in
\cite{GroechenigBook}  in which
both time and frequency are treated as equally important dimensions
for understanding signals.   
The symmetry between time and frequency also manifests in our error
estimates, and we believe that  the   unified
time-frequency  perspective  is a key reason why our estimates are sharp.

\subsection{Polynomial decay and finite smoothness}
Our main theorem will be proved for general  decay
conditions.  We first discuss the special case of  functions of polynomial decay and finite
smoothness. As remarked in ~\cite{Briggs}, finite smoothness is  often
replaced by polynomial decay of the \ft , which we will do here. The
first  error estimate is a simple consequence of our main result 
(Theorems~\ref{tm:WA} and 4.6)  and is the most relevant special case.  
\begin{theorem}\label{tm:intro0}
Let $a,b>1$. If $f$ and
$\hat{f}$ are continuous and 
satisfy polynomial decay of the form $\sup_{x\in\R}|f(x)|(1+|x|)^{a}<\infty$ and $\sup_{\xi\in\R}|\hat{f}(\xi)|(1+|\xi|)^{b}<\infty$, 
then, 
for all
$\alpha < a-\frac 1 2$ and $\beta <b-\frac 1 2$, there is a constant $c>0$ that is independent of $h$, $n$, $p$ such that 
\begin{equation}\label{eq:re1}
E^{[n]}_{h}(f) \leq c \,(  p^{-\alpha}+h^\beta)\,. 
\end{equation}
\end{theorem}

This statement can be read on several levels.
  \begin{enumerate}
  \item[(i)] It confirms  the users'  experience that the discrete \ft\
indeed approximates the continuous \ft .
\item[(ii)]   Theorem~\ref{tm:intro0} offers guarantees of how well the FFT
approximates the \ft\ by   quantifying  the approximation  error in terms of the decay of $f$ and
$\hat{f}$.  
\item[(iii)] The theorem provides a hint on how to
improve the accuracy of the numerical computations.  This is a new insight with possibly practical
consequences.\\
\indent If no a priori knowledge of $f$ is given, then the $n$ samples of $f$
are usually chosen on an interval of length $p = \sqrt{n}$ with
grid size $h=1/{\sqrt n}$. In this case the error decays like  $E^{[n]}_{h}(f) = \cO
(n^{-\min (\alpha , \beta )/2})$. However, with some knowledge about these
decay rates, one may balance the two terms in \eqref{eq:re1} so that
both terms contribute equally\footnote{We write
  $\lesssim$ if the left-hand-side is bounded by a constant times the
  right-hand-side. If $\lesssim$ and $\gtrsim$ both hold, then we
  write $\asymp$. The dependency of the constants on other parameters
  shall be clarified or is clear from the context.}, i.e.,
$p^{-\alpha}\asymp h^\beta\asymp n^{-\frac{\alpha
    \beta}{\alpha+\beta}}$.  For this choice  the error is
\begin{equation}\label{eq:re4}
E^{[n]}_{h}(f)\leq c  n^{-\frac{\alpha\beta}{\alpha +\beta}}  
\end{equation}
 This observation will become
relevant when  the decay rates $\alpha
, \beta$ are known and the difference   $|\alpha - \beta |$ is
large. The choice  $p^{-\alpha}\asymp h^\beta$ then leads to
a significantly more accurate numerical approximation of
  the  \ft\ using the same number of samples. Conversely, a given
  target  level of accuracy can be achieved with significantly fewer
  samples. See Section~\ref{sec:num}. \\
\indent We also note that this choice 
leads to  the asymptotically optimal relation between $h$, $p$, and
$n$, see Corollary \ref{cor:poly a and b}.
\item[(iv)]   In Corollary~\ref{cor:poly a and b  explicit} we will determine the constants involved in the balancing $p^{-\alpha}\asymp h^\beta\asymp n^{-\frac{\alpha
    \beta}{\alpha+\beta}}$, and we will determine the constant $c$ in \eqref{eq:re4} 
  explicitly in terms of $\alpha ,\beta $ and a norm on $f$. We
  emphasize that the ultimate version of the error estimate is a
  non-asymptotic result that may yield relevant information about the
  error even for small $n$.
  \end{enumerate}
Our numerical simulations in  Section \ref{sec:num}  confirm that the  bounds in
Theorem \ref{tm:intro0} are sharp up to some infinitesimal margin.

\paragraph*{Smoothness instead of Fourier decay}
One may argue that we also make certain assumptions on $\hat{f}$, which is
the very object we want to compute. The polynomial decay
condition of $\hat{f}$ in Theorem~\ref{tm:intro0} can be replaced by a stronger condition on smoothness of $f$.

Precisely, if the derivatives
$f^{(l)}$ are in $L^1(\R)$ for $l= 0, \dots ,  b  $ with $b\in\N$, then it follows that
$\sup_{\xi\in\R}|\hat{f}(\xi)|(1+|\xi|)^{b}<\infty$.   Using this
observation, we obtain a slightly weaker, but
equally  useful version of Theorem~\ref{tm:intro0}. 
\begin{corollary}
 Let $a,m>1$. Assume that  $|f(x)| \lesssim (1+|x|)^{-a}$ 
 and that $f$ has $m$ derivatives in $L^1(\R )$. Then for 
 $0< \alpha < a-1/2$ and $\beta < m-1/2 $  there is a constant $c>0$ 
 independent of $h$, $n$, $p$ such that 
\begin{equation} \label{eq:re3}
E^{[n]}_{h}(f) \leq c \,(  p^{-\alpha }+h^{\beta})\,. 
\end{equation}
\end{corollary}

\subsection{Sub-exponential decay} Exponential and sub-exponential decay in time and in frequency occur importantly
in the theory of test functions and distributions (under the name of
Gelfand-Shilov spaces or ultra test functions)~\cite{bjork66}. The
strongest form is exponential decay for which error estimates have
been treated before.   We offer the following version of the error.

 \begin{theorem}[sub-exponential decay]\label{tm:intro1}
Let $r,s>0$ and $0<\alpha,\beta\leq 1$. If $f$ and $\hat{f}$ are continuous 
and satisfy the decay conditions $\sup_{x\in\R}|f(x)|e^{r|x|^\alpha}<\infty$ and $\sup_{\xi\in\R}|\hat{f}(\xi)| e^{s|\xi|^\beta}<\infty$, then for all $r'<r$ and $s'<s$, 
there is a constant $c>0$ that is independent of $h$, $n$, $p$ such that 
\begin{equation*}
E^{[n]}_h(f) \leq c\,( e^{-r' (\frac{p}{2})^{\alpha}}+ e^{-s' (2h)^{-\beta}}) \,.
\end{equation*}
\end{theorem}
For $r'\to r$ and $s' \to s$ the error is balanced if  $e^{-r (\frac{p}{2})^{\alpha}}\asymp e^{-s (2h)^{-\beta}}$, which is satisfied for $h = n^{-\frac{\alpha}{\alpha+\beta}} \left(\tfrac{s}{r}\right)^{\frac{1}{\alpha+\beta}} 2^{\frac{\alpha-\beta}{\alpha+\beta}}$. This condition on $h$ is the optimal choice for the relation between $h$, $p,$ and $n$, see Corollary \ref{cor:subexp a and b}.

For the parameters $\alpha = \beta =1$, the assumption says that both
$f$ and $\hat{f}$ decay exponentially. The choice $p=\sqrt{n}$ and
$h=1/{\sqrt n}$ then yields the error estimate
$$
E^{[n]}_{h}(f)\leq c\, e^{-d \sqrt{n}} \, ,
$$
with a constant $d$ depending on $r,s,\alpha, \beta $. Up to possibly different constants, this  root-exponential convergence is compatible with
Stenger's result~\cite{StengerBook}.

\subsection{Mixed  conditions} We finally  provide bounds for mixed
conditions when $f$ possesses exponential decay, but finite smoothness.

\begin{theorem}[mixed decay]\label{thm:mixed intro}
  Let $r>0$ and $0<\alpha\leq 1$. If $f$ and $\hat{f}$ are continuous and satisfy $\sup_{x\in\R} |f(x)|e^{r|x|}<\infty$ and $\sup_{\xi\in\R} |\hat{f}(\xi)|(1+|\xi|)^b<\infty$ with $b>1$, then, for all $r'<r$ and $\beta<b-\frac{1}{2}$, there is a constant $c>0$ that is independent of $h$, $n$, $p$ such that 
\begin{equation}
E^{[n]}_{h}(f)\leq c\,( e^{-r' p/ 2 } + h^{\beta})\,.
\end{equation} 
\end{theorem}
For all practical purposes, the length of the sampled interval  $p$ is
sufficiently large, so that $e^{-r' p/ 2}$ is
negligible (and usually close to machine precision). Then  the
error  is dominated  by the polynomial term $h^\beta$.  

 The above estimate implies a  pointwise error of
  $h^\beta$ for $\beta < b-1/2$. The error estimate \eqref{eq:bound from briggs} from
\cite{Briggs} induces a bound for the $\ell^2$-error
$E^{[n]}_{h}(f)$ on the order of $h^\beta$ for
$\beta=b-\frac{1}{2}$, however, under  additional, stringent boundary conditions.

For comparison,   we formulate   the error estimate for the
  known  case of  functions with compact support.
\begin{corollary} \label{corintro2}
  Assume that $f$ has compact support in $[-\frac{p}{2},\frac{p}{2}]$ with $p=hn$ and is $m$-times
  differentiable, then for every $\epsilon >0$ there is a
  constant $c=c(\epsilon,p)$, such that
  $$
  E^{[n]}_{h}(f)\leq c\, n^{-(m-1/2-\epsilon)}\,  . 
$$
\end{corollary}
We note that the pointwise estimate $n^{-m}$ from \cite{Epstein} translates
into the estimate $ E^{[n]}_{h}(f) \lesssim n^{-(m-1/2)}$. Thus except
for the additional $\epsilon $ in the exponent,  we recover the well-known results for
compactly supported functions. The presence of $\epsilon$ is
plausible, because the assumption of differentiability is slightly
stronger than the polynomial decay of $\hat{f}$, and we derive 
Corollary~\ref{corintro2} from that weaker assumption (and thus
obtain a weaker conclusion).

\subsection{Interpolation}\label{sec:intro inter}
Software  
to plot the discrete Fourier transform
commonly performs interpolation of the computed data vector $\DFT\!
f_{h,n}$. The implicit
question is then what the samples of $\hat{f}(k/p)$ or their
approximations say about the global Fourier transform $\hat{f}$. In order to simulate this process, we use three
representative interpolation procedures  and derive bounds on their
deviation from $\hat{f}$ on the real line. One of them is based on the
cardinal sine function  $\sinc(x) = \frac{\sin(\pi x)}{\pi x}$ that is
widely used in time-frequency analysis, digital signal
  processing, and numerical analysis
  \cite{butzer-2,TascheBook,Stenger81,StengerBook}.   The computation of $\DFT\!
f_{h,n}$ gives rise to the approximation of $\hat{f}$ on the real line
by  
\[
\hat{f}(\xi) \approx 
\sqrt{p} \sum_{k\in[n]} (\mathcal{F}f_{h,n})(k) \sinc(p\xi-k)
=h \sum_{k,j\in[n]} f(hj) e^{-2\pi\mathrm{i}\frac{kj}{n}} \sinc(p\xi-k)\,.
\]
We will  consider 
  the $L^2$ norm, this is, we estimate
$$\|\hat{f} -h \sum_{k,j\in[n]} f(hj)
e^{-2\pi\mathrm{i}\frac{kj}{n}} \sinc(p\cdot-k) \|_{L^2(\R)}\,, $$
and will
show in Section~\ref{sec:bandlimited}, Theorem \ref{tm:finale again11}, that  this error
obeys exactly  the same bounds as $E_h^{[n]}(f)$ in Theorems \ref{tm:intro0},
\ref{tm:intro1}, and \ref{thm:mixed intro}.

\subsection{Minimal requirements for convergence}\label{sec:int:minimal}
Whereas in most applications and in scientific computing it is
important to understand how well the FFT approximates the Fourier
transform and how fast the convergence happens, it is a fundamental
question whether and when the approximations converge at
all. Mathematically, we search for minimal conditions on a function. 
To the best of our knowledge, this problem  has not  been addressed
in the literature. \\ 
\indent 
For the initial formulation, we use the pointwise error 
and weak polynomial decay conditions. 
\begin{proposition} \label{prop:c6}
  If $f$ and $\hat{f}$ satisfy the mild decay conditions
  $\sup_{x\in\R}|f(x)|(1+|x|)^{1+\epsilon}<\infty$ and
  $\sup_{\xi\in\R}|\hat{f}(\xi)|(1+|\xi|)^{1+\epsilon}<\infty$ for
  some $\epsilon >0$, then 
  \begin{equation}
\lim _{p\to \infty } \lim _{h\to0}
\sup_{k\in[n]}\Big|\hat{f}(\tfrac{k}{p}) - h\sum_{j\in[n]}
f(hj)e^{-2\pi\mathrm{i} \frac{k j}{n}} \Big| = 0 \,.
  \end{equation}
\end{proposition}
Clearly, these decay conditions are the weakest possible. If $\epsilon
= 0$, then $f$ need not be integrable and its Fourier transform may
not even be defined pointwise. 

In Section~\ref{sec:minimal} we will formulate a more general, but
technical condition that is related to the validity of the
Poisson summation formula,  and we will also treat the convergence of
the interpolation. 
   \newlength{\Breite}
 \setlength{\Breite}{.75\textwidth}

 \subsection{Other error measures and noisy samples}
We will be exclusively interested in the $\ell ^2$-approximation error
$E^{[n]}_h(f)$, that is the error made by replacing the Fourier
transform by the FFT. Other error norms or the
influence of other effects, such as  rounding or distortions,  can be derived  from
the main estimates.

(i) \emph{Pointwise  error estimates}  follow easily because 
\begin{equation}
  \label{eq:c9}
\max _{k\in [n]}  \Big| \hat{f}\Big( \tfrac{k}{p} \Big) - h \sum_{j \in [n]} f(hj)
  e^{-2\pi i \frac{kj}{n}} \Big| \leq \sqrt{p} E^{[n]}_h(f)  \, ,
\end{equation}
and other $\ell^q$-norms 
can be handled similarly. As an example,
under the decay assumptions of Theorem~\ref{tm:intro0},
\eqref{eq:re4} implies the pointwise error estimate 
\begin{align*}
\max _{\in [n]}  \Big| \hat{f}(\tfrac{k}{p}) - h\sum_{j\in[n]} f(hj)
  e^{-2\pi\mathrm{i}\frac{kj}{n}}\Big| 
  \lesssim n^{-\frac{\beta}{\alpha +\beta } (\alpha - \frac{1}{2})} \, ,
\end{align*}
 where the last inequality follows for the choice $p^{-\alpha } \asymp n^{-\frac{\alpha
    \beta}{\alpha+\beta}}$.
We do not claim
that the  bounds for  the pointwise error are optimal. Our  tools are
designed for $E^{[n]}_h(f)$ and  rely heavily on its time-frequency
symmetry, which is not shared by the pointwise error.

(ii) \emph{Noise, contaminated function evaluation, or round-off errors} are
important sources of errors in scientific computing. In each case, instead of the
correct samples $f(jh)$ the samples of a distorted function
$\widetilde{f}$ or noisy samples $\widetilde{f}(hj)= f(hj)+\epsilon _j$ are used as input of
the FFT. However, since the approximation \eqref{eq:c1} is linear and
the FFT is unitary, the effect of these errors is benign and remains
what it was. Formally, using $\widetilde{f}$ and its samples
$\widetilde{f}_{h,n}$ instead of $f_{h,n}$  as input of the FFT,  
the resulting error is 
\begin{equation}\label{eq:distortion}
 \big\| \hat{f}_{\frac{1}{p},n} - \DFT\! \widetilde{f}_{h,n}\big\|
 \leq  \big\| \hat{f}_{\frac{1}{p},n} - \DFT\! f_{h,n}\big\| +  \big\|\DFT\! f_{h,n}
 - \DFT\! \widetilde{f}_{h,n}\big\|
 = 
 E^{[n]}_h(f)+\big\| (\widetilde{f}-f)_{h,n}\big\|\,.
\end{equation}
Thus the error splits into the \emph{approximation error} $E^{[n]}_h(f)$, which we
are concerned with, and the norm of the initial noise.  
Consequently, all error bounds  remain valid for the distorted
function values  after adding the noise  term $\big\| (\widetilde{f}-f)_{h,n}\big\|$ to the original bound.

\subsection{Methods}
As in~\cite{Briggs} we will use the Poisson summation formula 
  to decompose  the approximation error into a time and a frequency
  component. This allows us to estimate each component separately.
 In order to make this approach precise, we need to sample a function
 on a grid and to apply  a pointwise version of Poisson’s summation formula. 
The novelty of our approach is (i) the systematic exploitation of the
time-frequency symmetry of the problem, and (ii) the use of so-called
Wiener amalgam spaces to describe sampling and decay properties.
This class  is the canonical function class when sampling is
involved \cite{Feichtinger00,Feichtinger2,Fournier} and is widely used
in Fourier analysis. We believe that it is also very useful and
convenient in numerical analysis. 

\section*{Outline}
 Section \ref{sec:num} is dedicated to numerical simulations
 to illustrate that the theoretical results of
   Theorems~\ref{tm:intro0} and \ref{thm:mixed intro} are best possible. 
The full versions of the theoretical results on the approximation rates
for spaces of polynomial and sub-exponential decay in time and in
frequency are derived in Section \ref{sec:WA}.
The approximation of $\hat{f}$ on the real line via interpolation  is discussed in Section
\ref{sec:bandlimited}.
In Section \ref{sec:minimal}  we investigate under which minimal conditions the error
  tends to zero.

\section*{Acknowledgements}
The authors thank Norbert Kaiblinger for valuable suggestions that helped improving the manuscript. 
We would like to express our gratitude to  the referees for their meticulous reading and
detailed comments on the manuscript. Their suggestions have helped us
immensely  to improve the manuscript. 
A.\ K.\ is supported by the FWF project DISCO (PAT4780023). 

\bigskip\bigskip

\section{Numerical simulations}\label{sec:num}
We back up the theoretical results with  numerical evidence that the
polynomial error bounds in Theorem \ref{tm:intro0} are sharp,  up to some infinitesimal margin. 

It is important to emphasize that the following numerical examples are
not intended to represent typical functions for a specific
application. The sole purpose of this section is to illustrate that
the assumptions and error rates associated with polynomial decay
cannot, in general, be improved. Strictly speaking, our theoretical
results allow for arbitrary $\alpha < a - \frac{1}{2}$, while the
numerics suggest that the rate $\alpha = a - \frac{1}{2}$ still
holds. This is what we mean by `sharp up to some infinitesimal
margin'. 

This numerical verification of optimality  is nontrivial: it requires
functions $f$ for which both $f$ and its Fourier transform $\hat{f}$
can be evaluated with high accuracy, and for which the decay rates in
both time and frequency are precisely known. For functions with a
standard closed-form expression, at least one of these properties
typically fails. We therefore construct suitable functions through a
more refined approach below.

\subsection{Polynomial decay in time and in frequency}
In view of Theorem \ref{tm:intro0}, we construct
a family $f^{a,b}$
of functions with exact polynomial decay $a$ in time and $b$ in frequency
such that 
each function and its Fourier transform can be numerically evaluated with sufficient accuracy. 

We use infinite linear combinations of shifts of the cardinal  B-splines   
defined by  $B_1:={\bf 1}_{[-\frac{1}{2},\frac{1}{2})}$, and 
\begin{equation*}
B_{b+1}:=B_{b}*B_1,\qquad b=1,2,\ldots\,.
\end{equation*}
 Each $B_b$ is a piecewise polynomial function of degree $b-1$ that is
 $b-2$-times continuously differentiable with  support 
 $[-\frac{b}{2},\frac{b}{2}]$, and its  Fourier transform is 
 $\widehat{B_b}(\xi ) = \big( \tfrac{\sin \pi \xi }{\pi \xi } \big)^b
 = \sinc ^b (\xi )$, see \cite{Schoenberg} and also
 \cite[Section 9.1]{TascheBook}.  

Consider $(c_k)_{k\in\Z}\subseteq\C$ given by its nonzero entries
\begin{equation*}
c_0=\frac{1}{2},\qquad c_k = \frac{1}{k\pi \mathrm{i}}, \text{ for }
k\in 2\Z+1\, , 
\end{equation*}
so that $c_{2k}=0$.
For integer parameters $a,b \in \N$  we study the family of  functions
\begin{equation}\label{eq:fab}
f^{a,b}(x) := \sum_{k\in \Z} c^a_k \,B_b(x-k),\qquad a,b=1,2,\ldots \, .
\end{equation}
Every $f^{a,b}$ is a  locally finite sum,  therefore its  point
evaluations can be computed accurately in numerical experiments, and they satisfy $\sup_{x\in\R}|f^{a,b}(x)| (1+|x|)^{a}<\infty$.  

To compute the Fourier transform of $f^{a,b}$, we define the Fourier
series $u_a(\xi):=\sum_{k\in \Z} c^a_ke^{-2\pi\mathrm{i}k\xi}$ and
obtain 
\begin{align*}
\widehat{f^{a,b}}(\xi) 
& = \sinc(\xi)^b \,  \sum_{k\in \Z} c^a_ke^{-2\pi\mathrm{i}k\xi}=\sinc(\xi)^b u_a(\xi)\,.
\end{align*}
The estimate $|\sinc^b(\xi)| \lesssim (1+|\xi|)^{-b}$ implies
$\sup_{\xi\in\R}|\widehat{f^{a,b}}(\xi)|(1+|\xi|)^{b}<\infty$, thus $f^{a,b}$
satisfies the  assumptions of Theorem~\ref{tm:intro0}.

By a short computation we find $u_1(\xi) = {\bf
  1}_{(-\frac{1}{2},0)}(\xi)+\frac{1}{2}{\bf
  1}_{\{-\frac{1}{2},0\}}(\xi)$  for $\xi\in
[-\frac{1}{2},\frac{1}{2})$, and then  $u_{a+1}$ is the cyclic convolution of $u_a$ and $u_1$, i.e., 
\begin{equation*}
u_{a+1}(\xi)=\int_{\R/\Z} u_a(\eta)u_1(\xi-\eta)\mathrm{d}\eta , \qquad a=1,2,\ldots\,.
\end{equation*}
For $a=2,3,4,5$, tedious calculations lead to periodic functions with
period $1$ whose values on $[-1/2,1/2)$ are given by 
\begin{align*}
u_2(\xi)  &= |\xi|,\\
u_3(\xi) &= -\sign(\xi) \xi^2+\frac{1}{2}\xi+\frac{1}{8},\\
u_4(\xi) &= 2|\xi|^3/3 - \xi^2/2  + 1/12,\\
u_5(\xi) &= -\sign(\xi)\xi^4/3 + \xi^3/3 - \xi/24+1/32 \, . 
\end{align*}
Thus, we can compute the point evaluations of $f^{a,b}$ and
$\widehat{f^{a,b}}$ exactly, both analytically and numerically, and
the  decay parameters  are precisely   $a$ and $b$. 
For the numerical simulations we compute the samples
$\widehat{f^{a,b}}(k/p)$ and compare them with the discrete Fourier
transform of a sampled version of $f^{a,b}$. 

To compute the discrete Fourier transform, we choose  $n=2^l$,
$l=10,\ldots, 20$ and then  apply the FFT to the samples $f^{a,b}_{h,n}$. The FFT is taken from the Julia package \texttt{FFTW.jl}  
that provides a binding to the \texttt{C} library FFTW \cite{FFTW}.

Since the error estimates of Theorem \ref{tm:intro0} work  for every
parameter $\alpha < a-1/2$ and $\beta < b-1/2$, we choose $\alpha=a-\frac{1}{2}$ and
$\beta=b-\frac{1}{2}$. We compare the results of the  numerical
experiments  with the
error rates $n^{-\frac{\alpha\beta}{\alpha+\beta}}$ predicted by Theorem \ref{tm:intro0} for
$h=n^{-\frac{\alpha}{\alpha+\beta}}$ that are listed in Table \ref{table:unique}. 
\begin{table}[h]
\centering
\begin{tabular}{ c c | c| c |c| c|c}
 $a$ & $b$  &$\alpha=a-\frac{1}{2}$ & $\beta=b-\frac{1}{2}$ &
                                                              $h=n^{-\frac{\alpha}{\alpha+\beta}}$&
                                                                                                    $p=nh=n^{\frac{\beta}{\alpha+\beta}}$ &  $E^{[n]}_h(f^{a,b}) \asymp h^\beta=n^{-\frac{\alpha\beta}{\alpha+\beta}}$\\ 
 \hline 
$2$ & $2$ & $3/2$ & $3/2$ &$n^{-1/2}$  & $n^{1/2}$ & $n^{-3/4}$ \\  
$2$ & $3$ & $3/2$ & $5/2$& $n^{-3/8}$ & $n^{5/8}$ & $n^{-15/16}$ \\  
$2$ & $4$ & $3/2$ & $7/2$& $n^{-3/10}$ & $n^{7/10}$ & $n^{-21/20}$ \\
$3$ & $2$ & $5/2$ & $3/2$& $n^{-5/8}$ & $n^{3/8}$ & $n^{-15/16}$ \\  
$3$ & $3$ & $5/2$ & $5/2$& $n^{-1/2}$ & $n^{1/2}$ & $n^{-5/4}$ \\  
$3$ & $4$ & $5/2$ & $7/2$& $n^{-5/12}$ & $n^{7/12}$ & $n^{-35/24}$\\
\end{tabular}
\caption{List of the relevant values of Theorem \ref{tm:intro0} for $f^{a,b}$, so that $E^{[n]}_h(f^{a,b})$ is bounded by $h^\beta=p^{-\alpha}=n^{-\frac{\alpha\beta}{\alpha+\beta}}$.}\label{table:unique}
\end{table}

Our numerical experiments as illustrated in Figure \ref{fig:3} align
perfectly with the theoretical bounds in Theorem \ref{tm:intro0} as presented in Table \ref{table:unique} and suggest that they are sharp up to some infinitesimal margin.

\begin{figure}
\centering
\includegraphics[width=\Breite]{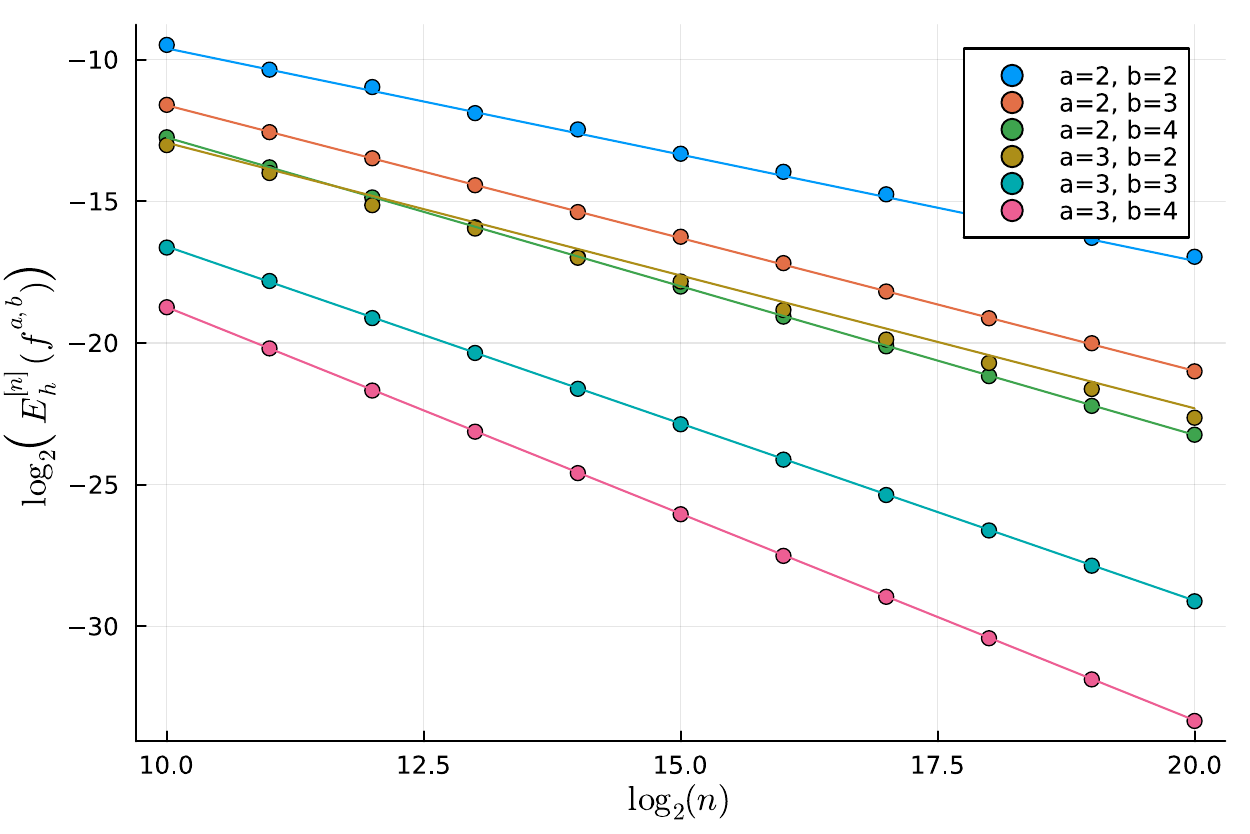}
\caption{Logarithmic plots of $E^{[n]}_{h}(f^{a,b})$ for various $a$
  and $b$. Points are computed numerically for $n=2^l$, $l=10,\ldots,
  20$. Reference lines represent the  theoretical slopes predicted by
  Theorem \ref{tm:intro0} and listed in Table \ref{table:unique}.
}\label{fig:3}
\end{figure}

\begin{remark}
Let us revisit the construction of $f^{a,b}$ to highlight why it is so
remarkable that we are able to perform numerical experiments. The
B-splines $B_b$ and their Fourier transforms $\sinc^b$ can be
evaluated accurately. Because B-splines have compact support, the
infinite series defining $f^{a,b}$ reduces to a finite sum when
evaluated at a finite number of points. Moreover, the careful choice
of coefficients in this expansion ensures that the Fourier transform
of $f^{a,b}$ admits a closed-form expression and  enables the
accurate numerical evaluation. This is a rare situation, since  most
functions do not satisfy these properties, and meaningful numerical
experiments depend critically on them. 

In particular, the closed-form expressions for $
u_a(\xi) = \sum_{k \in \mathbb{Z}} c^a_k e^{-2\pi i k \xi}
$ 
are essential to our numerical results. Without them, the series in
\eqref{eq:fab} would need to be truncated, introducing truncation
errors. Summing many terms also risks floating-point inaccuracies,
reducing the reliability of computed values of
$E^{[n]}_h(f^{a,b})$. The availability of a closed-form expression for
$u_a$ eliminates these sources of numerical error. 
\end{remark}

\subsection*{Choice of $p$ and $h$}
To highlight the importance of a good choice of  $p$ and $h$, we consider the function $f^{2,4}$, $f^{2,5}$, $f^{3,5}$, $f^{5,2}$, whose decay in time differs from the decay in frequency. We 
compare two parameter selections: the optimal choice  suggested by
Theorem \ref{tm:intro0}, versus the standard choice $p = \sqrt{n}$ and
$h = 1/\sqrt{n}$~\cite{boyd,Kaiblinger05}. As shown in Figures
\ref{fig:00}, \ref{fig:001}, \ref{fig:002}, and \ref{fig:003}, the
optimal parameters yield an error decay rate consistent with our
theoretical predictions and clearly outperform the standard choice.

From the plots one can draw valuable information about the number of
samples required to achieve a target accuracy. For instance, by drawing a horizontal
line in Figure~\ref{fig:00} at the error $E_h^{[n]}(f^{2,4})=
2^{-15}\approx 3\cdot 10^{-5}$ we find that $n=2^{12} $ samples are
required for the optimal choice of $h,n,p$, whereas $2^{17}$ samples
are required for the choice $p=h^{-1}=\sqrt{n}$. Even for an accuracy
of only five digits, this is an enormous saving!

\begin{figure}
\centering
\includegraphics[width=\Breite]{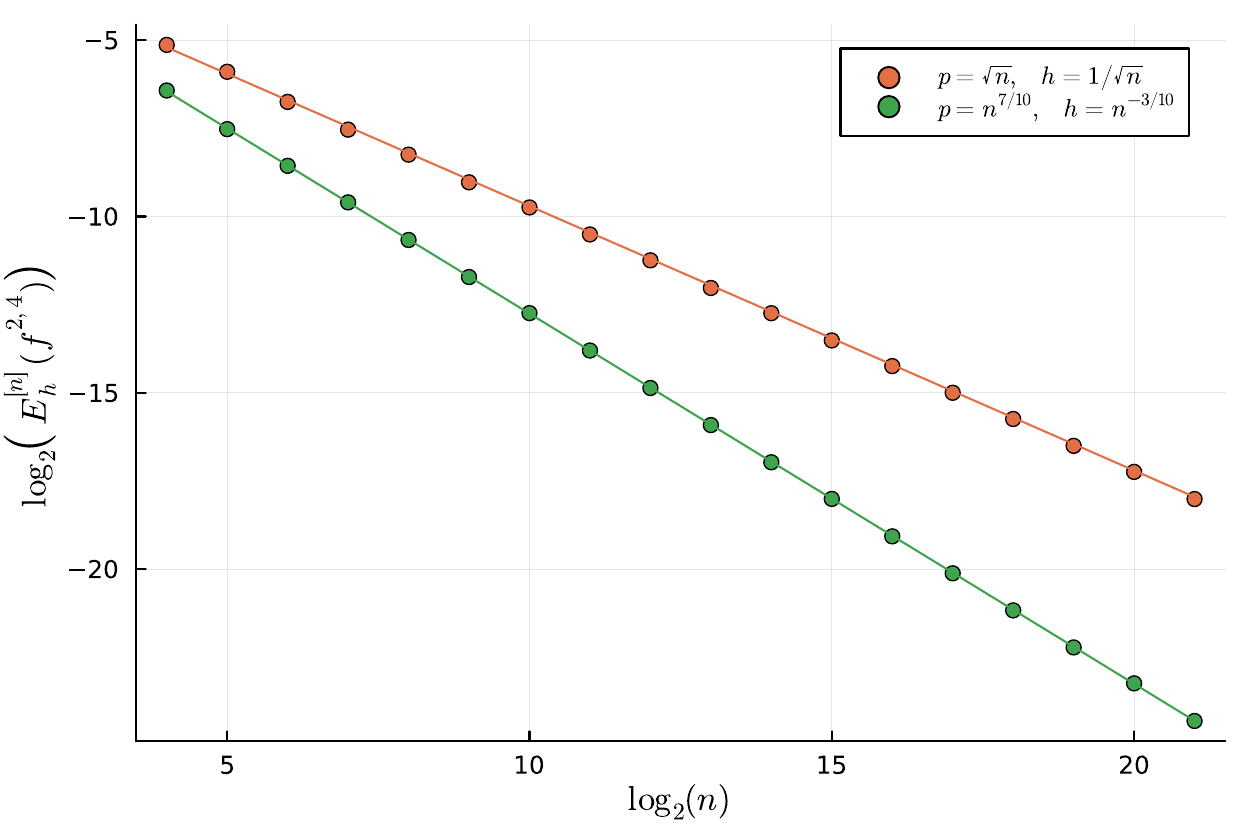}
\caption{For $f^{2,4}$, we compare the error curve  for
    the  standard choice $p = \sqrt{n}$, $h = 1/\sqrt{n}$ with the
    error curve for the improved choice $p = n^{7/10}$, $h =
    n^{-3/10}$. The reference lines indicate the theoretical error
    decay rates: $n^{-3/4}$ for the standard choice and $n^{-21/20}$
    for the optimal  one.
}\label{fig:00}
\end{figure}

\begin{figure}
\centering
\includegraphics[width=\Breite]{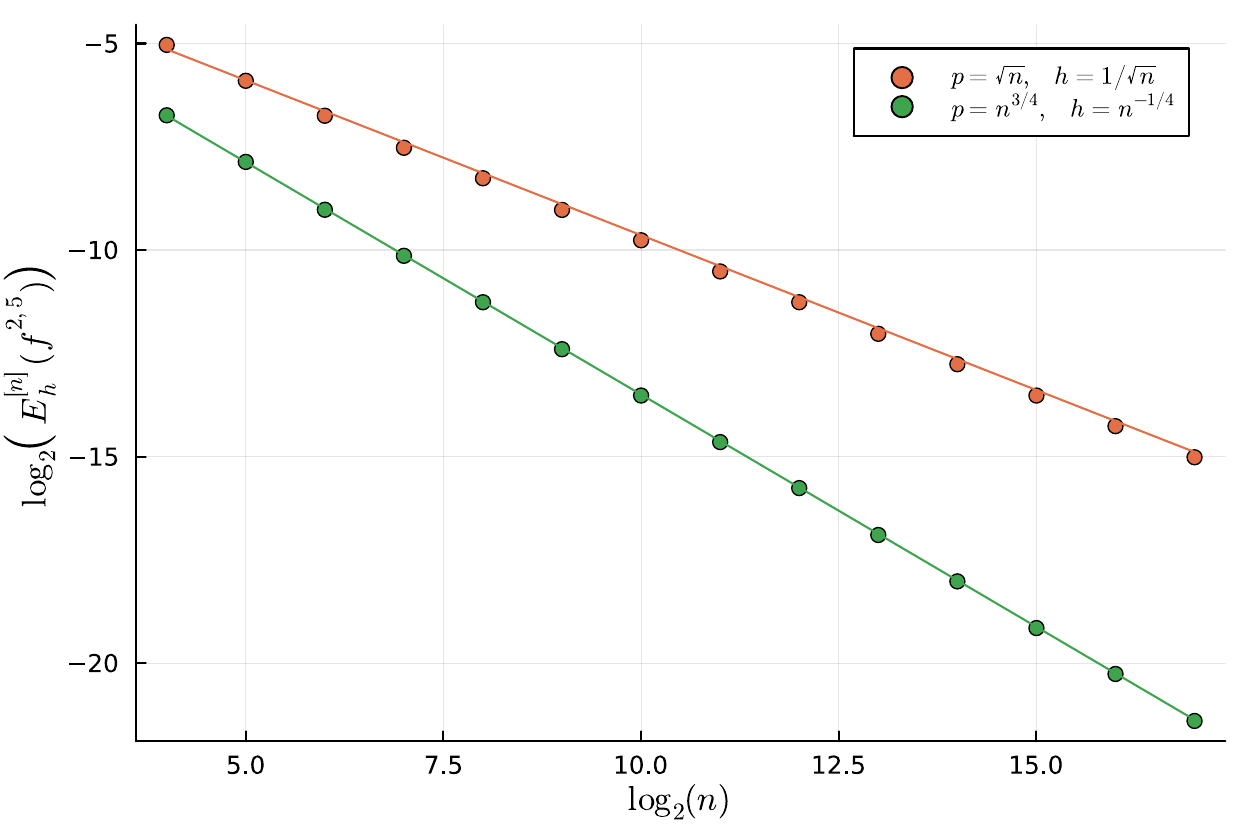}
\caption{Comparison of error curves for $f^{2,5}$. 
}\label{fig:001}
\end{figure}

\begin{figure}
\centering
\includegraphics[width=\Breite]{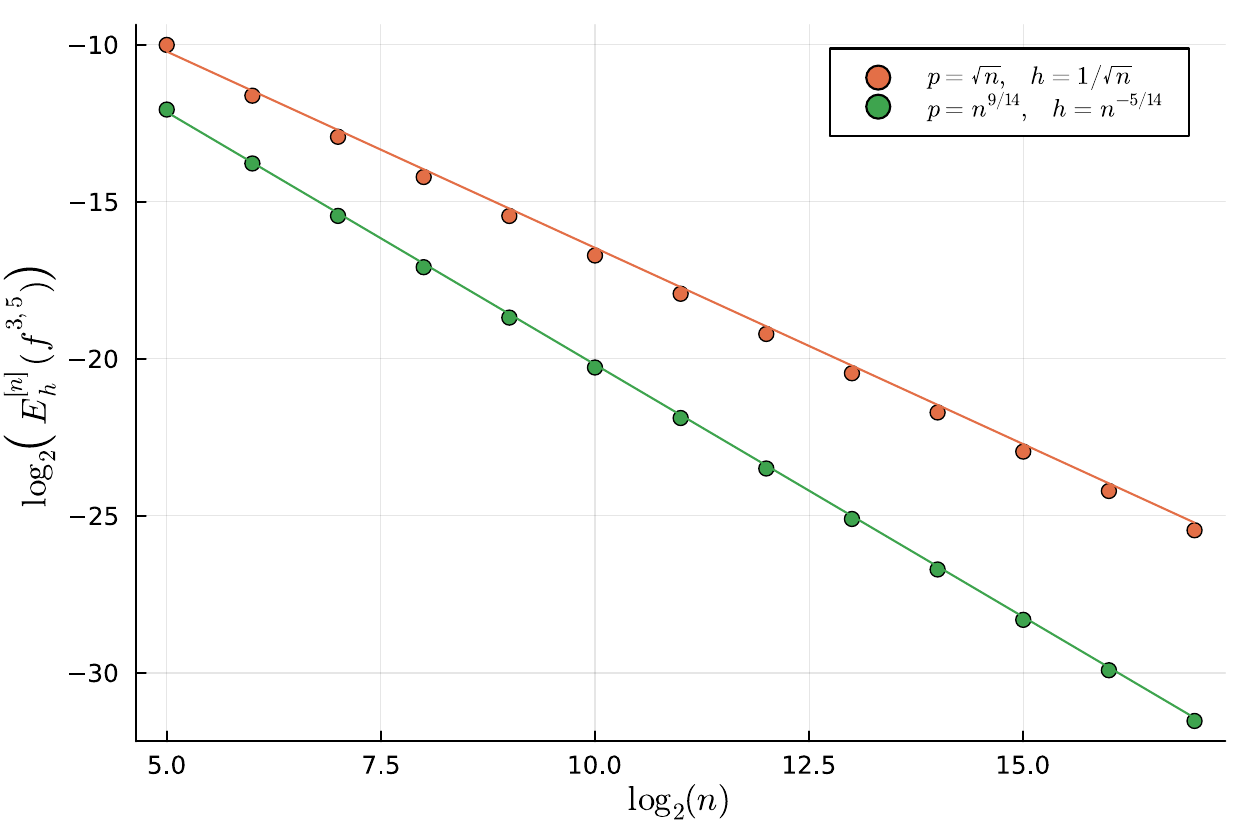}
\caption{Comparison of error curves for $f^{3,5}$.
}\label{fig:002}
\end{figure}

\begin{figure}
\centering
\includegraphics[width=\Breite]{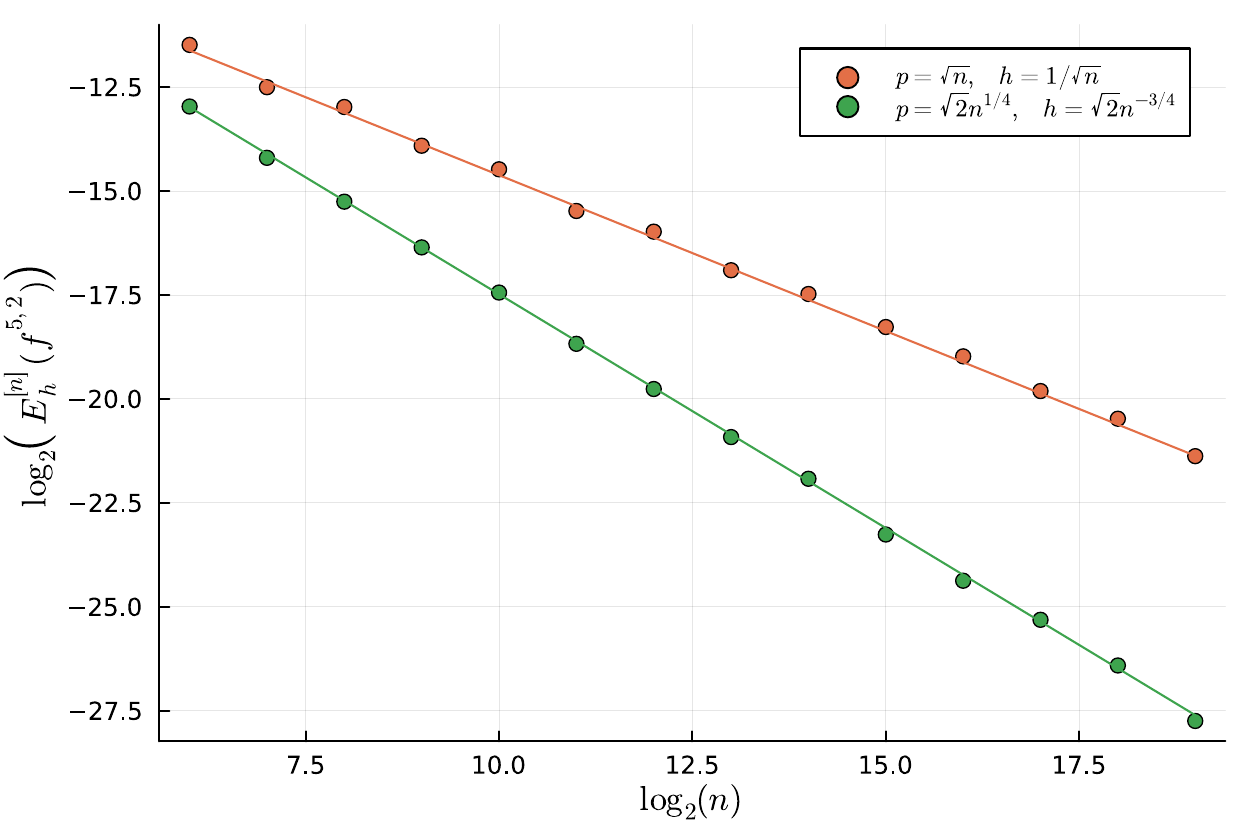}
\caption{Comparison of error curves for $f^{5,2}$.
}\label{fig:003}
\end{figure}

\subsection{Exponential decay in time and polynomial decay in frequency}\label{sec:exp in time and pol in freq}
To demonstrate that Theorem \ref{thm:mixed intro} is best possible, we
consider the exponential  functions
$$
f_\delta(x) = e^{-2\pi|x - \delta|}, 
$$
for a shift parameter  $\delta \in \mathbb{R}$. Its Fourier transform is
$$
\hat{f_\delta}(\xi) = \frac{e^{-2\pi i \delta \xi}}{\pi(1 + \xi^2)}.
$$
Thus, $f_\delta$ decays exponentially in space and polynomially in frequency. This corresponds to the parameters in Theorem \ref{thm:mixed intro} being $r = 2\pi$, $\alpha = 1$, and $\sup_{\xi \in \mathbb{R}} |\hat{f_\delta}(\xi)|(1 + |\xi|)^b < \infty$ with $b = 2$.

When $p$ is sufficiently large, the error in Theorem \ref{thm:mixed intro} is dominated by the polynomial decay over a wide range of $n$. For numerical experiments using the FFT with $\delta = 0$, we consider the choices
$$
p = n^{1/2}, \quad p = 6n^{1/10}, \quad \text{and} \quad p = 10\,,
$$
corresponding to $h = n^{-1/2}$, $h = 6n^{-9/10}$, and $h = 10n^{-1}$,
respectively. The expected convergence rates are $h^{3/2}$, which
translate into  
$$
n^{-3/4},\quad n^{-27/20},\quad \text{and} \quad n^{-3/2}\,.
$$
Figure \ref{fig:40} confirms these predictions with excellent
agreement between theory and experiment for \emph{all} shift
parameters $\delta $ (the plots look identical for $0\leq \delta\leq 1$).

\begin{figure}
\centering
\includegraphics[width=\Breite]{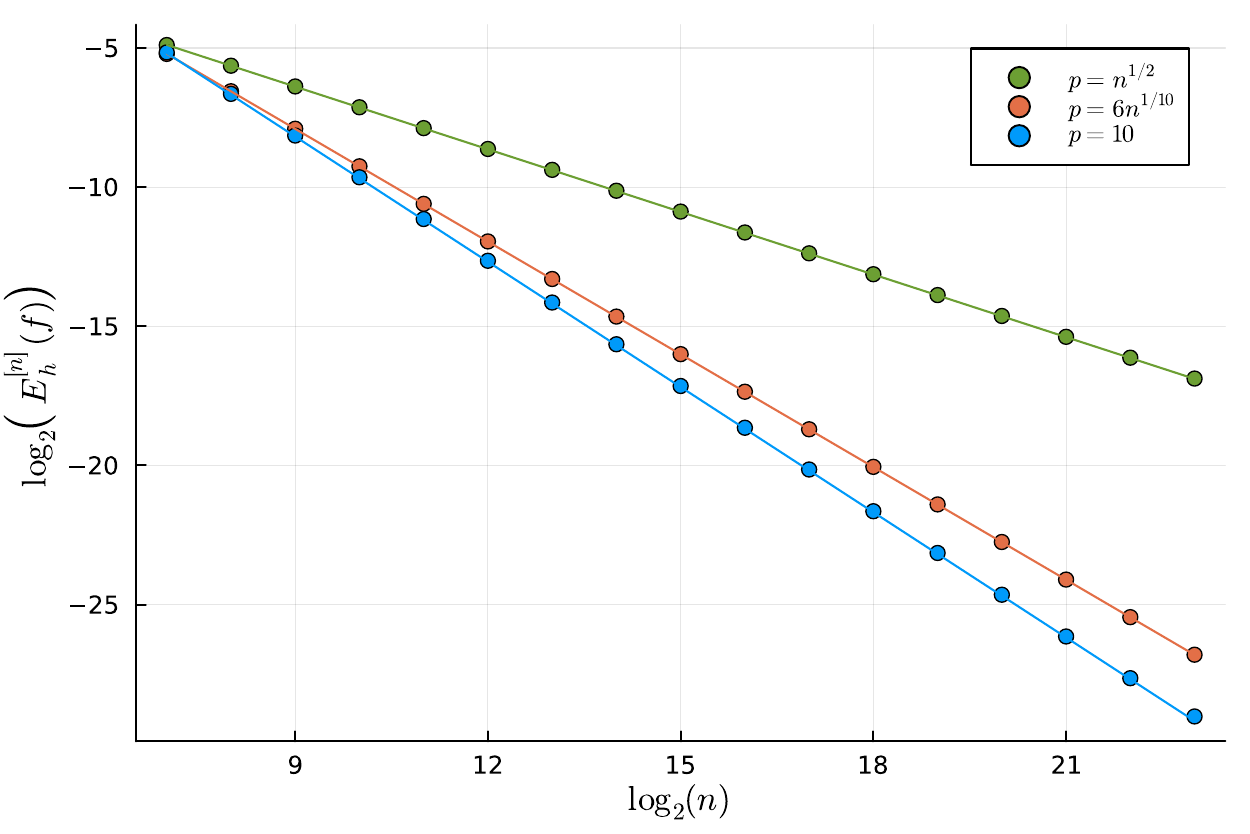}
\caption{Logarithmic plots of $E^{[n]}_{h}(f)$ for $f(x)=e^{-2\pi|x|}$. Points are computed numerically for $n=2^l$, $l=7,\ldots, 23$. We choose $p=n^{\frac{1}{2}}$, $p=6n^{1/10}$, and $p=10$. Reference lines represent the theoretical slope $h^\beta$ with $\beta=\frac{3}{2}$ that lead to $n^{-\frac{3}{4}}$, $n^{-\frac{27}{20}}$, and $n^{-\frac{3}{2}}$, respectively. Points align well with theoretical predictions.}\label{fig:40}
\end{figure}

This example is also considered in a case study in
\cite{Briggs}. For $\delta = 0$ the pointwise error is calculated to
be of the order $n^{-2}$. However, as soon as  $\delta \neq 0$  and $p
= 10$, say, the required condition $f_\delta(p/2) = f_\delta(-p/2)$ in
\cite{Briggs} is violated, and their bounds yield only $n^{-1}$ for
the maximum norm, and hence only $n^{-1/2}$ for the $2$-norm as used
in $E^{[n]}_h(f)$, whereas both Theorem~\ref{thm:mixed intro} and our numerical
experiments consistently  yield the error rate  $n^{-3/2}$. 
This further supports the sharpness of our theoretical bounds.

\section{Approximation rates}\label{sec:WA}
In this section we  prove Theorems \ref{tm:intro0}, \ref{tm:intro1},
and \ref{thm:mixed intro} from the introduction. All these statements
follow from a single theorem by feeding it with various decay
conditions. Specifically, we derive   error bounds for the averaged
$\ell ^2$-error  $E^{[n]}_h(f)=\| \hat{f}_{\frac{1}{p},n} - \DFT\!
f_{h,n}\|$. This error exhibits a  striking symmetry between time and
frequency, which will  be reflected in the symmetry of the error
estimates. 
\begin{lemma}\label{lemma:symmetry}
If  $p=nh$ and $f,\hat{f}\in L^1(\R)$, then 
\begin{equation}\label{eq:E sym}
E^{[n]}_h(f) = E^{[n]}_{p^{-1}}(\bar{\hat{f}})\,.
\end{equation}
\end{lemma}
\begin{proof}
  To verify \eqref{eq:E sym}, we recall that
  $(\bar{\hat{f}})\,\widehat{}=  \bar{f}$ and that $\mathcal{F}$ is unitary, so that
\begin{align*}
E^{[n]}_{p^{-1}}(\bar{\hat{f}}) & 
                                  = \|\bar{f}_{h,n}-\mathcal{F}\bar{\hat{f}}_{\frac{1}{p},n} \| =  \|\mathcal{F}^*\bar{f}_{h,n}-\bar{\hat{f}}_{\frac{1}{p},n} \| \,.
                                  \end{align*}
The discrete Fourier transform also satisfies $\mathcal{F}^*\bar{y}=\overline{\mathcal{F}
  y}$, where complex conjugation is meant entry-wise. Therefore, we derive $E^{[n]}_{p^{-1}}(\bar{\hat{f}}) = \|\overline{\mathcal{F}f_{h,n}}-\bar{\hat{f}}_{\frac{1}{p},n} \|$. Complex conjugation does not affect the norm, so that the latter coincides with $\|\mathcal{F}f_{h,n}-\hat{f}_{\frac{1}{p},n} \|=E^{[n]}_h(f)$.
\end{proof}
Thus, we expect good bounds for $E^{[n]}_h(f)$ to reflect the same symmetry between $(f,h)$ and $(\hat{f},p^{-1})$.

\subsection{Error decomposition into time and frequency components}\label{sec:decomp}
We derive general and explicit bounds on the approximation error $E^{[n]}_h(f)=\| \hat{f}_{\frac{1}{p},n} - \DFT\! f_{h,n}\|$ in \eqref{eq:error original unique} by decomposing the error into a time and a frequency component that are estimated separately. These components arise from
approximating  $\hat{f}(\tfrac{k}{p})$ first by a Riemann sum and then truncating it  as in \eqref{eq:1approx}, 
\begin{equation}\label{eq:quad and trunc}
\hat{f}\big(\tfrac{k}{p}\big) \approx h\sum_{j\in\Z} f(hj)
e^{-2\pi\mathrm{i}\frac{k}{p}hj } \approx h\sum_{j\in[n]} f(hj)
e^{-2\pi\mathrm{i} \frac{kj}{n} } = \sqrt{hn} (\DFT\! f_{h,n})(k) \,.
\end{equation}
The left-hand side is the sampled (continuous) Fourier transform, the
right-hand side is its approximation by the discrete Fourier
transform. The idea is to estimate the discretization error 
\begin{equation*}
\hat{f}\big(\tfrac{k}{p}\big) - h\sum_{j\in\Z} f(hj)
e^{-2\pi\mathrm{i}\frac{k}{p}hj }
\end{equation*}
by the  decay  of $\hat{f}$, and the truncation error 
\begin{equation*}
h\sum_{j\in\Z \setminus [n]} f(hj)
e^{-2\pi\mathrm{i}\frac{k}{p}hj }
\end{equation*}
by the decay of $f$.

\subsection{Wiener amalgam spaces} \label{sec:wamalg}
The treatment of the error makes use of sampling  a function on a
grid and a strong pointwise version of
the Poisson summation formula. To handle these, we introduce  the
 Wiener amalgam spaces that are  tailored to these objectives. 

The continuous functions on the real
  line are denoted by $\mathcal{C}=\mathcal{C}(\R)$.  
 The space $W(\mathcal{C},\ell^1)$ consists of all continuous functions $f$ such that 
\begin{equation}\label{eq:def WA0}
\|f\|_{W(\mathcal{C},\ell^1)} := \sum_{l\in\Z} \sup_{x\in[0,1]} \left| f(x+l)\right|
\end{equation}
is finite. Wiener amalgam spaces were introduced by N.\
Wiener~\cite{wiener} and have become convenient and important function
spaces in Fourier analysis,
cf.~\cite{Feichtinger00,Feichtinger2,Fournier}. The norm in
  \eqref{eq:def WA0} is easy to understand. A function $f$  belongs to
  $W(\mathcal{C}, \ell ^1)$, if it is continuous and is majorized by an
  integrable   step function with jumps at the integers. 
  See  Figure~\ref{fig:wam}.
\begin{figure}
\centering
\includegraphics[width=\Breite]{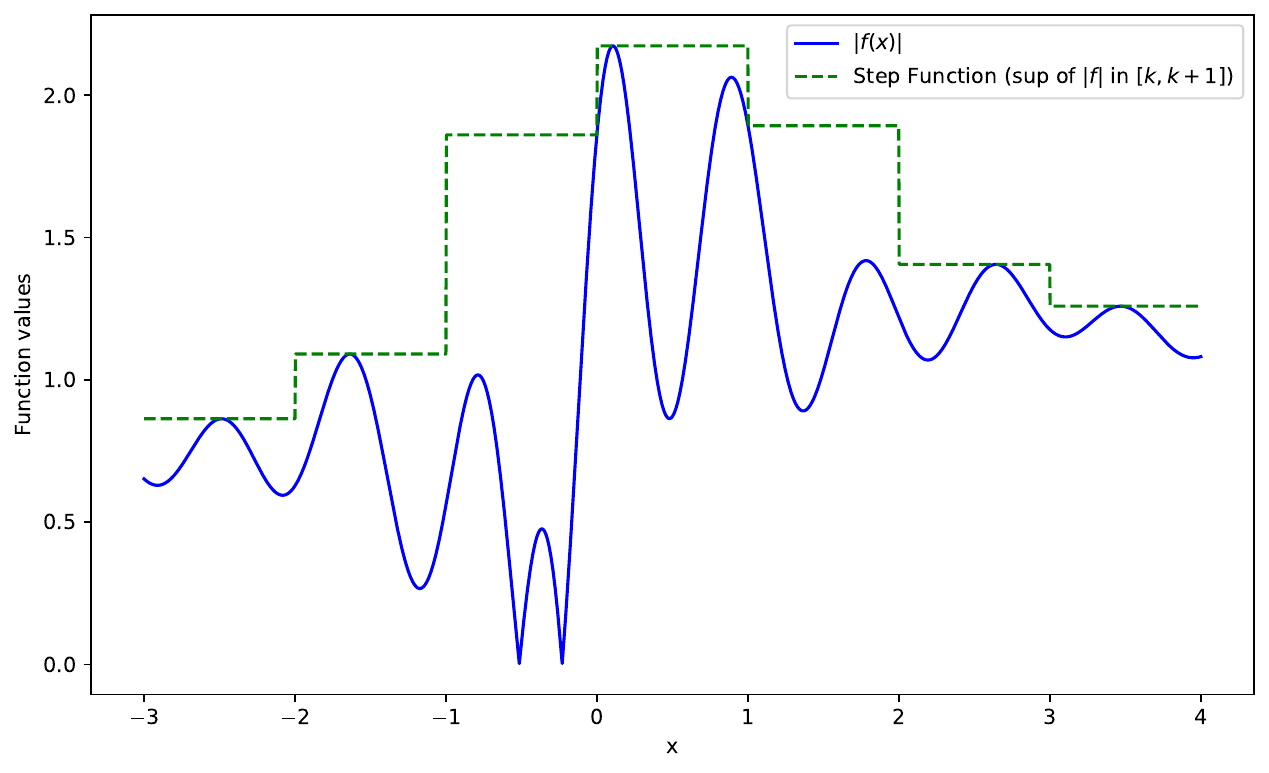}
\caption{Visualization of the construction of the Wiener amalgam norm
  of a continuous function $f$.  
}\label{fig:wam}
\end{figure} \\
\indent In the context of the discretization of the \ft\ this space arises
  naturally, because both sampling and periodization are well defined. 
If $f\in W(\mathcal{C},\ell ^1)$, then the sampling operator $f \mapsto
(f(hj))_{j\in\Z}$ maps $W(\mathcal{C},\ell ^1)$ to $\ell ^1(\mathbb{Z})$  for all $h>0$,  and~\footnote{Here  $\lceil\cdot\rceil$ denotes the ceiling  function.}
$$
  \sum _{j\in \Z } |f(jh)| \leq \lceil \tfrac{1}{h}\rceil
  \|f\|_{W(\mathcal{C},\ell ^1)} \, ,
  $$
  Consequently the $p$-periodization 
\begin{equation*}
\cPp f (x)= \sum_{l\in\Z} f(x+pl)
\end{equation*}
converges absolutely and uniformly for every $p>0$. 

Furthermore, if  $f,\hat{f} \in W(\mathcal{C}, \ell ^1)$, then  the
Poisson summation formula 
\begin{equation*}
\sum _{l\in \Z} f(x+pl) = 
\frac{1}{p}\sum _{k\in \Z} \hat{f}(\tfrac{k}{p}) e^{2\pi i \frac{k}{p}x}
\end{equation*}
holds pointwise for
\emph{all} $x\in \R $, cf.~\cite[Lemma 4]{gro99}.  No simple,
  larger space seems to be known on which the Poisson summation
  formula holds pointwise for all periods $p$.

The following identity  provides the foundation of our approach and
will facilitate the error decomposition in a time and a frequency component. 

\begin{lemma}\label{lemma:op norm sampling00}
If $f,\hat{f}\in W(\mathcal{C},\ell^1)$, then, for $k\in[n]$, 
\begin{equation}\label{eq:2 important identities}
  h\sum_{j\in\Z} f(hj) e^{-2\pi\mathrm{i}\frac{k}{p} hj}
  = h\sum_{j\in[n]}(\cPp f)(hj)e^{-2\pi\mathrm{i}\frac{kj}{n}}=   (\cPh \hat{f})(\tfrac{k}{p})\,.
\end{equation}
\end{lemma}
We may also write the second identity as $\mathcal{F} (\cPp f)_{h,n}=(\cPh \hat{f})_{\frac{1}{p},n}$. 
\begin{proof}
  For the first identity, we  partition $\Z$ into residue classes modulo  $n$ and  use
$p=hn$. Then 
\begin{align*}
h\sum_{j\in\Z} f(hj) e^{-2\pi\mathrm{i}\frac{k}{p}hj} & = h\sum_{j\in[n]} \sum_{l\in\Z} f(hj+hnl) e^{-2\pi\mathrm{i}\frac{k}{hn}(hj+hnl)}\\
& = h\sum_{j\in[n]} \sum_{l\in\Z} f(hj+pl) e^{-2\pi\mathrm{i}\frac{k}{hn}hj}\\
& = h\sum_{j\in[n]} \cPp f (hj) e^{-2\pi\mathrm{i}\frac{k}{n}j}\,.
\end{align*}
The second identity in \eqref{eq:2 important identities} follows from
the Poisson summation formula applied to the $h \inv $-periodization
of  $\hat{f}$, 
\begin{align*}
(\cPh \hat{f})(\tfrac{k}{p}) & = h \sum_{j\in\Z}  \hat{\hat{f}}(hj)e^{2\pi\mathrm{i}hj\frac{k}{p}}\\
& = h \sum_{j\in\Z}  f(hj)e^{-2\pi\mathrm{i}hj\frac{k}{p}}\,.
\end{align*}
This is again the left-hand-side of \eqref{eq:2 important identities}.
\end{proof}
Lemma \ref{lemma:op norm sampling00} enables a reinterpretation of the two step approximation strategy via quadrature and truncation \eqref{eq:quad and trunc}. The quadrature step compares the samples of the Fourier transform $\hat{f}(\frac k p)$ to samples of the periodization of the Fourier transform $(\cPh \hat{f})(\frac k p)$,
 which coincide with the discrete Fourier transform  of samples of the
 periodization of $f$, i.e.,  $\mathcal{F} (\cPp f)_{h,n}(k )$. The
 truncation compares the latter  with the discrete Fourier transform
 $\DFT\! f_{h,n}$. Symbolically,  
\begin{equation}\label{eq:instruction}
\hat{f}_{\frac{1}{p},n}\approx (\cPh \hat{f})_{\frac{1}{p},n}=\mathcal{F} (\cPp f)_{h,n} \approx \DFT\! f_{h,n}\,.
\end{equation}
We therefore split the pointwise error in two  components,
\begin{equation}
  \label{eq:pointwise}
|\frac{1}{\sqrt{p}}  \hat{f}(\tfrac k p) -  \DFT\! f_{h,n} (k)| \leq
  \frac{1}{\sqrt{p}} |\hat{f} (\tfrac k p )- \cPh \hat{f} (\tfrac k p
  ) | + |\DFT  \big((\cPp f- f)_{h,n}\big)(k)|\,.
\end{equation}
Taking norms and using that $\DFT$ is unitary, we arrive at 
\begin{equation}\label{eq:bound E tf}
  \begin{split}
  E^{[n]}_h(f) = \|\hat{f}_{\frac{1}{p},n} - \DFT\! f_{h,n}\|
  \leq &\,  \|(\hat{f} - \cPh \hat{f})_{\frac{1}{p},n} \| + \|\DFT
 (\cPp f- f)_{h,n} \| \\
  =&  \,\|(\hat{f} - \cPh \hat{f})_{\frac{1}{p},n} \| + \|(\cPp f- f)_{h,n}\|\,.
  \end{split}
\end{equation}
On the right-hand side of this inequality we have (i) eliminated
$\DFT$ and (ii) split the error in a time component and a frequency
component. Note that the symmetry of this decomposition  matches the
symmetry of the error  $E^{[n]}_h(f) = E^{[n]}_{p^{-1}}(\bar{\hat{f}})$ in Lemma \ref{lemma:symmetry}.  
To derive approximation rates, we must ensure that the samples of $f-\cPp f$ and $\hat{f}-\cPh \hat{f}$ decay sufficiently fast. This can be accomplished by specifying decay conditions on $f$ and $\hat{f}$.

To handle the sampling of a function on a grid simultaneously with  its decay properties, we introduce
 another family of Wiener amalgam spaces. We first quantify the decay of a function by the use of a weight function $v:\R\rightarrow (0,\infty)$. We consider weights $v$ that satisfy\footnote{The condition \eqref{eq:d} is crucial, whereas the conditions \eqref{eq:a}, \eqref{eq:b}, and \eqref{eq:c} can be weakened at the expense of more complicated constants in the error estimates.}
\begin{subequations}\label{eq: weights}
  \begin{align}
   &v \text{ is even: } v(-x)=v(x)\,,  \label{eq:a}\\
   &v\text{ is nondecreasing on the positive axis, i.e., } v(x)\leq v(y), \text{ for } |x|\leq |y| \,, \label{eq:b}\\ 
   &v\text{ is submultiplicative, i.e., } v(x+y)\leq v(x) v(y) \text{ for } x,y \in \R \,, \label{eq:c}\\
   &v^{-1} \text{ is square-summable, } \sum_{l\in\Z}{|v(l)|^{-2}}<\infty \,. \label{eq:d}
  \end{align}
\end{subequations}
The standard  examples of weights that satisfy conditions \eqref{eq:a} to \eqref{eq:d}  are the polynomial weights $v_\alpha(x)=(1+|x|)^\alpha$
for $\alpha>\frac{1}{2}$, and the (sub-)exponential weights
$v_{r,\alpha}(x)=e^{r|x|^\alpha}$, for $0<\alpha\leq 1$ and $r>0$.

Associated to $v$ are  the weighted sequence spaces
\[
\ell^q_v:=\ell^q_v(\Z):=\{ (c_l)_{l\in\Z} \subseteq \C: \sum_{l\in\Z}
|c_l|^q |v(l)|^q <\infty\}\, , 
\]
and the weighted Lebesgue space $L^2_v:=L^2_v(\R)$  of all functions $f$ such that 
$
\|f\|_{L^2_v} =\|f v\|_{L^2} 
$
is finite. For the constant weight $v \equiv 1$ we omit the
subscript and write $\ell ^p$ and $L^p$.

The Wiener amalgam space $W(\mathcal{C},\ell^q_v)$ consists of all continuous functions $f$ such that 
\begin{equation}\label{eq:def WA}
\|f\|_{W(\mathcal{C},\ell^q_v)} := \bigg(\sum_{l\in\Z} \sup_{x\in[0,1]} \left| f(x+l)\right|^q |v(l)|^q\bigg)^{1/q}
\end{equation}
is finite. For the theory and applications of Wiener amalgam spaces we
recommend ~\cite{Feichtinger00,Feichtinger2,Fournier}. 

If $v$ satisfies \eqref{eq:d} then
$\ell^2_v\subseteq \ell^1$, and we obtain the continuous embeddings
$W(\mathcal{C},\ell^2_v)\subseteq
W(\mathcal{C},\ell^1)\subseteq L^1(\R)$.

\subsection{Towards the main error estimates} \label{sec:mainthm}

In order to bound the right-hand side of the error functional $E^{[n]}_h(f)$  in  \eqref{eq:bound E tf}, we need to estimate the deviation of $f$ and $\hat{f}$ from their respective periodizations. We derive such bounds in terms of the function
\begin{equation}\label{eq:phi}
\Phi_v(p):=\bigg(2\sum_{m=0}^\infty |v(p m+\tfrac{p}{2})|^{-2}\bigg)^{1/2}\,.
\end{equation}
Observe that $\lim_{p \to \infty}\Phi_v(p)= 0$ always holds for a weight  $v$ satisfying \eqref{eq: weights}. For polynomial and sub-exponential weights we will obtain quantitative bounds on  the decay of $\Phi_v(p)$ in Lemma~\ref{lemma:bound poly} and in Lemma~\ref{lemma:Phi subexp}.
\begin{lemma}\label{lemma:op norm sampling}
If $f\in W(\mathcal{C},\ell^2_v)$  for a weight $v$ satisfying \eqref{eq: weights}, then
\begin{align*}
\tag{i}\|f-\cPp f\|_{L^2(-\frac{p}{2},\frac{p}{2})} & \leq v(1)\Phi_v(p)\|f\|_{W(\mathcal{C},\ell^2_v)},\\
\tag{ii} \|(f-\cPp f)_{h,n}\|  & \leq  v(1) \Phi_v(p) (1+h)^{\frac{1}{2}} \|f\|_{W(\mathcal{C},\ell^2_v)}\,.\label{eq:bound Phi}
\end{align*}
\end{lemma}
Part (i) is  used only  in Section \ref{sec:bandlimited} on
interpolation, but Part (ii) is crucial for bounds on $E^{[n]}_h$.  
\begin{proof}
(i) We bound $(f-\cPp f)(x)= \sum_{0\neq l\in\Z} f(x+pl)$ by the Cauchy-Schwarz inequality,
\begin{equation}\label{eq:x hj 0}
\bigg|\sum_{0\neq l\in\Z} f(x+pl) \bigg|^2
  \leq \sum_{0\neq m\in\Z} |v(x+pm)|^{-2}\sum_{0\neq l\in\Z} \left|f(x+pl) \right|^2|v(x+pl)|^2\,.
 \end{equation}
Since $v$ is even \eqref{eq:a} and nondecreasing \eqref{eq:b}, for
$|x|\leq \frac{p}{2}$, the first sum can be majorized  by
\begin{equation}\label{eq:x hj}
\begin{split}
\sum_{0\neq m\in\Z}|v(x+pm)|^{-2} & \leq \sum_{0\neq m\in\Z} |v(p|m|-\tfrac{p}{2})|^{-2}\\
&\leq 2\sum_{m=1}^\infty |v(pm-\tfrac{p}{2})|^{-2} = |\Phi_v(p)|^2\,.
\end{split}
\end{equation}
Integration of \eqref{eq:x hj 0} yields
\begin{align*}
\int_{-\frac{p}{2}}^{\frac{p}{2}}\bigg|\sum_{0\neq l\in\Z} f(x+pl) \bigg|^2\mathrm{d}x & \leq |\Phi_v(p)|^2 \int_{-\frac{p}{2}}^{\frac{p}{2}} \sum_{0\neq l\in\Z} \left|f(x+pl) \right|^2|v(x+pl)|^2\mathrm{d}x\\
& = |\Phi_v(p)|^2 \int_{|x|>\frac{p}{2}} \left|f(x) \right|^2|v(x)|^2\mathrm{d}x\\
& \leq |\Phi_v(p)|^2 \|f\|^2_{L^2_v}\,.
\end{align*}
Direct computations lead to $\|f\|_{L^2_v}\leq v(1)\|f\|_{W(\mathcal{C},\ell^2_v)}$. 

(ii) We use \eqref{eq:x hj 0} and \eqref{eq:x hj} for $x=hj$ and sum over $j\in[n]$, so that $p=nh$ yields 
 \begin{align*}
 \sum_{j\in[n]} \bigg( \sum_{0\neq l\in\Z}\left| f(hj+pl)\right|\bigg)^2 & \leq |\Phi_v(p)|^2 \sum_{j\in[n]}  \sum_{0\neq l\in\Z} \left|f(hj+pl) \right|^2 |v(hj+pl)|^2\\
  & = |\Phi_v(p)|^2  \sum_{j\in\Z\setminus[n]} \left|f(hj) \right|^2 |v(hj)|^2\\
  & \leq |\Phi_v(p)|^2  \sum_{j\in\Z} \left|f(hj) \right|^2 |v(hj)|^2\,.
 \end{align*}
Since $v$ satisfies \eqref{eq:a} -- \eqref{eq:c}, we
derive
  \begin{align*}
\sum_{j\in\Z} \left|f(hj) \right|^2 |v(hj)|^2
&\leq  \lceil h^{-1}\rceil \sum_{l\in\Z} \sup_{x\in[l,l+1]}\left|f(x)\right|^2|v(|l|+1)|^2\\
&\leq \lceil h^{-1}\rceil |v(1)|^2\sum_{l\in\Z} \sup_{x\in[0,1]}\left|f(x+l)\right|^2|v(l)|^2\,.
\end{align*}
The observation $h \lceil h^{-1}\rceil \leq (1+h)$ concludes the proof. 
\end{proof}

\subsection{The main theorem} \label{sec:mainresult}

After these preparations we can now formulate our main result. This is
an error estimate for the approximation of the Fourier transform by
the discrete Fourier transform under very general conditions on
the decay of the function and its Fourier transform. 

\begin{theorem}\label{tm:WA}
Let  $f\in W(\mathcal{C},\ell^2_v)$ and $\hat{f}\in
W(\mathcal{C},\ell^2_w)$. If the  weights $v,w$ satisfy \eqref{eq: weights}, then 
\begin{equation}\label{eq:weigtsrest22}
E^{[n]}_{h}(f)\leq 
v(1)\Phi_v(p) (1+h)^{\frac 1 2} \|f\|_{W(\mathcal{C},\ell^2_{v})}+
w(1)\Phi_w(h^{-1}) (1+\tfrac 1 p)^{\frac1 2} \|\hat{f}\|_{W(\mathcal{C},\ell^2_{w})}\,. 
\end{equation}
\end{theorem}
For  $0<h \leq 1 \leq p$, we obtain
\begin{equation}
  \label{eq:weigtsrest}
  E^{[n]}_{h}(f)\leq 
2^{\frac 1 2}v(1)\Phi_v(p) \|f\|_{W(\mathcal{C},\ell^2_{v})}+
2^{\frac 1 2}w(1)\Phi_w(h^{-1})  \|\hat{f}\|_{W(\mathcal{C},\ell^2_{w})}\,,
\end{equation}
so that the first term does not depend on $h$
and the second term not on $p$. 

We now prove Theorem \ref{tm:WA}.
\begin{proof}
Recall the bound $
E^{[n]}_{h}(f) \leq \| (\hat{f} -\cPh \hat{f})_{\frac{1}{p},n}\|+\|  (f - \cPp f)_{h,n}\|
$  in \eqref{eq:bound E tf}. We apply  Lemma \ref{lemma:op norm sampling} (ii) to both $f$ and $\hat{f}$. 
Lemma \ref{lemma:op norm sampling} estimates $\|  (f - \cPp
f)_{h,n}\|$ and yields the first term of the error in
\eqref{eq:weigtsrest22}. Applying Lemma~\ref{lemma:op norm sampling} to
$\| (\hat{f} -\cPh \hat{f})_{\frac{1}{p},n}\|$ and interchanging  the
roles of $f$, $p$, and $h$ with  $\hat{f}$, $h^{-1}$, and $p^{-1}$,
yields the second term. This concludes the proof.
\end{proof}

Theorem \ref{tm:WA} asserts a  bound on the error $E^{[n]}_h(f)$ in
terms of $\Phi_v(p)$ and $\Phi_w(h^{-1})$. For useful conclusions we
need to unravel the error function $\Phi _v$, and we will do this   
for polynomial and sub-exponential weights.   The resulting  error
estimates can then  be
made completely  explicit.

\subsection{Polynomial weights} \label{sec:poly weight}
In this section we  consider the  polynomial weights $v_\alpha(x)=(1+|x|)^\alpha$ and derive bounds on $\Phi_{v_\alpha}(p)$ in \eqref{eq:phi}. 
\begin{lemma}\label{lemma:bound poly}
If $\alpha>\frac{1}{2}$, then 
\begin{equation*}
\Phi_{v_\alpha}(p) 
  \leq p^{-\alpha} 2^{\alpha +1/2}
\Big(1+\frac{1}{4\alpha -2}\Big)^{1/2} \lesssim p^{-\alpha}\, .
\end{equation*}
\end{lemma}
\begin{proof}
We factor out $p$ to obtain
\begin{align*}
 |\Phi_{v_\alpha}(p)|^2=2\sum_{m=0}^\infty
 \left(1+\left|pm+\tfrac{p}{2}\right|\right)^{-2\alpha}  &\leq p^{-2\alpha}
2 \sum_{m=0}^\infty \left(\tfrac{1}{p}+m+\tfrac{1}{2}\right)^{-2\alpha}\\ 
& \leq p^{-2\alpha}2\sum_{m=0}^\infty \left(m+\tfrac{1}{2}\right)^{-2\alpha}\,.
\end{align*}
This sum has the shape of the Hurwitz zeta function
$\zeta(s,t)=\sum_{m=0}^\infty (m+t)^{-s}$ at $t=1/2$. The integral test for convergence of series provides, for $\alpha >1/2$,
  \begin{align*}
    \zeta(2\alpha,\tfrac 1 2)  \leq 2^{2\alpha} + \int_0^\infty (x+\tfrac{1}{2})^{-2\alpha}\mathrm{d}x
     = 2^{2\alpha} (1+ \tfrac{1}{4\alpha-2})\,.
  \end{align*}
\end{proof}
For polynomial weights Theorem \ref{tm:WA} takes the following shape.

\begin{theorem}[polynomial weights]\label{cor:poly a and b}
Assume that  $f\in W(\mathcal{C},\ell^2_{v_\alpha})$ and $\hat{f}\in W(\mathcal{C},\ell^2_{v_\beta})$ for $\alpha,\beta>\frac{1}{2}$, and $0<h\leq 1\leq p$. 
\begin{itemize}
\item[(i)] Set $c_s = 2^{2s +1} \left(1+\tfrac{1}{4s-2}\right)^{\frac 1 2}$. Then 
\begin{equation*}
    E^{[n]}_{h}(f)\leq
    c_\alpha p^{-\alpha}  \|f\|_{W(\mathcal{C},\ell^2_{v_\alpha})}+
    c_\beta
    h^\beta \|\hat{f}\|_{W(\mathcal{C},\ell^2_{v_\beta})}\,.
\end{equation*}
\item[(ii)] If the step size satisfies $h\asymp n^{-\frac{\alpha}{\alpha+\beta}}$, then 
\begin{equation*}
E^{[n]}_{h}(f) \lesssim h^\beta
\big(\|f\|_{W(\mathcal{C},\ell^2_{v_\alpha})}+\|\hat{f}\|_{W(\mathcal{C},\ell^2_{v_\beta})}\big)\,
,
\end{equation*}
where the constant in $\lesssim$ is independent of $h,n,p,f$, but may depend on $\alpha,\beta$. 
Expressed with the number of samples, the error is $$
E^{[n]}_{h}(f)
\lesssim n^{-\frac{\alpha\beta}{\alpha +\beta }}
$$
with a constant independent of $h,n,p$. 
\end{itemize}
\end{theorem}

Part (ii) can be made more explicit by collecting all constants, leading to the ideal choice of $h$ in relation to properties of $f$.
\begin{corollary}[explicit constants]\label{cor:poly a and b explicit}
Under the assumptions of Theorem \ref{cor:poly a and b}, we make the
choice  
\begin{equation*}
h =   \left(\frac{c_\alpha  \|f\|_{W(\mathcal{C},\ell^2_{v_\alpha})} }{c_\beta \|\hat{f}\|_{W(\mathcal{C},\ell^2_{v_\beta})}} \right) ^{\frac{1}{\alpha+\beta}}n^{-\frac{\alpha}{\alpha+\beta}}\,,
\end{equation*}
where $c_s = 2^{2s +1} \left(1+\tfrac{1}{4s-2}\right)^{\frac 1 2}$. Then the resulting error bound becomes
\begin{equation*}
E^{[n]}_{h}(f) \leq 2 \left(c_\alpha \|f\|_{W(\mathcal{C},\ell^2_{v_\alpha})}\right)^{\frac{\beta}{\alpha+\beta}}  \left(c_\beta \|\hat{f}\|_{W(\mathcal{C},\ell^2_{v_\beta})}\right)^{\frac{\alpha}{\alpha+\beta}} n^{-\frac{\alpha\beta}{\alpha+\beta}}\,.
\end{equation*}
\end{corollary}
This bound fully matches the symmetry $E^{[n]}_{h}(f) = E^{[n]}_{p^{-1}}(\bar{\hat{f}})$ recognized in Lemma \ref{lemma:symmetry}. 

\begin{proof}
(i) The assumption $h\leq 1\leq p$ implies that $\sqrt{1+h} \leq
\sqrt{2}  $ and $\sqrt{1+p^{-1}} \leq
\sqrt{2}  $. We note $v_\alpha (1) = 2^\alpha $. Using  Lemma
\ref{lemma:bound poly}, the constant $v_\alpha (1)  \Phi _{v_\alpha
}(p) \sqrt{1+h}$ in Theorem~\ref{tm:WA} is then bounded by $2^{2\alpha
  +1} (1+1/(4\alpha -2))^{1/2}$, likewise for the constant involving
$w=v_\beta$.  The error estimate now   follows from Theorem
\ref{tm:WA}. 

(ii) We balance the terms $h^\beta $ and $p^{-\alpha }$. Since
$h=\frac{p}{n}$,  the choice $h\asymp
n^{-\frac{\alpha}{\alpha+\beta}}$  leads to $h^{\beta} \asymp
p^{-\alpha}$ and the overall error is of order $h^\beta $.  
\end{proof}
Theorem  \ref{cor:poly a and b} and Corollary \ref{cor:poly a and b explicit} answer the questions (Q1) and (Q2)
raised in the introduction whenever  $f$ decays polynomially in time
and in frequency.

\begin{remark}
In the literature, e.g.~\cite{boyd,Kaiblinger05}, one can find  the
choice $h\asymp \frac{1}{\sqrt{n}}$. It is optimal for 
$\alpha=\beta$. If $\alpha\neq \beta$, then $h\asymp
n^{-\frac{\alpha}{\alpha+\beta}}$ leads to better bounds. 
That is, knowledge on the decay of $f$ and $\hat{f}$ allows us to
refine the spacing, and Theorem~\ref{cor:poly a and b} 
explains how to choose the optimal parameters.  
\end{remark}

In our experience, amalgam spaces are the natural theoretical framework when dealing with
sampling and periodization. Yet one may wish to have conditions that
are easier to check in practice. One such condition is a pure decay
condition of the form $\sup_{x\in\R}|f(x)|(1+|x|)^{a}<\infty$ as in Theorem \ref{tm:intro0}. 
\begin{proof}[Proof of Theorem \ref{tm:intro0}]
Assume that $\sup_{x\in\R}|f(x)|(1+|x|)^{a}<\infty$  and
$\alpha<a-\frac{1}{2}$. We show that $f\in W(\mathcal{C}, \ell
^2_{v_\alpha })$. H\"older's inequality yields
\begin{equation} \label{later1}
\sum_{l\in\Z} \sup_{x\in[0,1]} \left| f(x+l)\right|^2 (1+|l|)^{2\alpha}  \lesssim \sup_{x\in\R}|f(x)|^2(1+|x|)^{2a}\sum_{l\in\Z}(1+|l|)^{-2(a-\alpha)}\,.
\end{equation}
Since  $a-\alpha>\frac{1}{2}$, the series
$\sum_{l\in\Z}(1+|l|)^{-2(a-\alpha)}<\infty$ converges, so that 
\begin{equation*}
    \|f\|_{W(\mathcal{C},\ell^2_{v_\alpha})} \lesssim \sup_{x\in\R}|f(x)|(1+|x|)^{a},
\end{equation*}
and likewise for $\hat{f}$. Therefore Theorem~\ref{tm:intro0} is a
special case of  Theorem~\ref{cor:poly a and b}. 
\end{proof}

\subsection{Sub-exponential weights}
Next, we  consider  sub-exponential weights of the form
$v_{r,\alpha}(x)=e^{r|x|^\alpha}$.
 We first find a
bound for $\Phi_{v_{r,\alpha}}(p)$ in \eqref{eq:phi}. We directly derive
\begin{equation}\label{eq:subexp 00}
|\Phi_{v_{r,\alpha}}(p)|^2 = 2\sum_{m=0}^\infty e^{-2r\left(pm+\frac{p}{2}\right)^\alpha}    = e^{-2r(\frac{p}{2})^\alpha}2\sum_{m=0}^\infty e^{-2r(\frac{p}{2})^\alpha \left((2m+1)^\alpha-1\right)}\,.
\end{equation}
For fixed $0<\alpha\leq 1$, the series $\sum_{m=0}^\infty e^{-2r(\frac{p}{2})^\alpha \left((2m+1)^\alpha-1\right)}$ converges for every $p>0$. As a function of $p$, it is monotonically decreasing, so that 
\begin{equation}\label{eq:bound subexp}
|\Phi_{v_{r,\alpha}}(p)|^2 \leq e^{-2r(\frac{p}{2})^\alpha}2\sum_{m=0}^\infty e^{-2r(\frac{1}{2})^\alpha \left((2m+1)^\alpha-1\right)}\,,
\end{equation}
for $p\geq 1$. 
To derive more explicit bounds,
we use
the incomplete Gamma function
$\Gamma(s,u)=\int_u^\infty t^{s-1}e^{-t}\mathrm{d}t$.
\begin{lemma}\label{lemma:Phi subexp}
Assume that $0<\alpha\leq 1$ and $r>0$. 
\begin{itemize}
\item[(i)] If $p\geq 1$, then
\begin{equation*}
\Phi_{v_{r,\alpha}}(p) \leq e^{-r (\frac{p}{2})^\alpha}
\Big(2+\frac{1}{\alpha}\frac{e^{r2^{1-\alpha}}}{(r2^{1-\alpha})^{1/\alpha}} \Gamma(\tfrac{1}{\alpha},r2^{1-\alpha})\Big)^{1/2}\,.
\end{equation*}
\item[(ii)] If $p\geq \big(\frac{2^\alpha}{2r\alpha}\big)^{1/\alpha}$, then 
\begin{equation*}
\Phi_{v_{r,\alpha}}(p) \leq e^{-r (\frac{p}{2})^\alpha}  \Big(2+\frac{2^{\alpha}}{2\alpha^2 rp^\alpha}\Big)^{1/2} \leq e^{-r (\frac{p}{2})^\alpha} 
\Big(2+\frac{1}{\alpha}\Big)^{1/2}\,.
\end{equation*}
\end{itemize}
\end{lemma}
Note that the factor $\big(2+\frac{2^{\alpha}}{2\alpha^2 rp^\alpha}\big)^{1/2}$ tends to $\sqrt{2}$ when $p\rightarrow\infty$. 
\begin{proof}
According to \eqref{eq:subexp 00}, we must bound the series $\sum_{m=0}^\infty e^{-2r(\frac{p}{2})^\alpha \left((2m+1)^\alpha-1\right)}$. To verify Part (ii), consider $p\geq\big(\frac{2^\alpha}{2r\alpha}\big)^{1/\alpha}$. For $u:={2r(\frac{p}{2})^\alpha}$, the integral test for convergence of series yields 
\begin{align*}
\sum_{m=0}^\infty e^{-u \left((2m+1)^\alpha-1\right)} & \leq 1+\int_0^\infty e^{-u\left((2x+1)^\alpha-1\right)} \mathrm{d}x = 1+\frac{1}{2}e^u  \int_1^\infty e^{-u x^\alpha}\mathrm{d}x\,.
\end{align*}
A direct check reveals that $-\frac{1}{\alpha} u^{-\frac 1\alpha}
\Gamma(\frac{1}{\alpha},x^\alpha  u)$ is a primitive of $x\mapsto
e^{-u x^\alpha}$ and therefore $\int _1^\infty e^{-ux^\alpha}\,
\mathrm{d}x = \frac{1}{\alpha} u^{-\frac 1\alpha}
\Gamma(\frac{1}{\alpha},  u)$. 
Hence, we derive
\begin{equation}\label{eq:z bound}
\sum_{m=0}^\infty e^{-u \left((2m+1)^\alpha-1\right)} \leq 1+\frac{1}{2\alpha}\frac{e^u}{u^{1/\alpha}} \Gamma(\tfrac{1}{\alpha},u)\,. 
\end{equation}
The condition $p\geq \big(\frac{2^\alpha}{2r\alpha}\big)^{1/\alpha}$
implies $u =2r(\frac{p}{2})^\alpha  \geq \frac{1}{\alpha}\geq 1$. In that case, the incomplete Gamma function is bounded by 
\begin{equation}\label{eq:pinelis}
\Gamma(\tfrac{1}{\alpha},u)\leq \frac{1}{\alpha} u^{\frac{1}{\alpha}-1} e^{-u}\,,
\end{equation}
see \cite[Proposition 2.7 and Page 1275]{Pinelis}.
For completeness we reproduce the elementary argument: 
 set $a=  \alpha\inv$. The function  $h(t)=(a-1)\ln t-t$ is strictly concave, so it is majorized by its tangent $h_u(t)=h(u)+h'(u)(t-u)$ at $u$. 
Therefore
\begin{equation*}
	\Gamma(a,u)=\int_u^{\infty}e^{h(t)}\,dt
	<\int_u^{\infty}e^{h_u(t)}\,dt =\frac{u^{a-1} e^{-u}}{1-(a-1)/u} \leq a {u^{a-1} e^{-u}}\,,
\end{equation*}
where the last inequality is due to $(1- (a-1)/u) \inv \leq a$ for   $u \geq a \geq 1$. This yields \eqref{eq:pinelis}. 

We substitute \eqref{eq:pinelis} into \eqref{eq:z bound}, use the
assumption $u\geq  1/\alpha $, and  obtain
\begin{align*}
\sum_{m=0}^\infty e^{-u\left((2m+1)^\alpha-1\right)}  & \leq  1+\frac{1}{2\alpha^2u}\leq
        1+\frac{1}{2\alpha}\,.
\end{align*}
Inserting $u={2r(\frac{p}{2})^\alpha}$ and the assumption $p^\alpha\geq \frac{2^\alpha}{2r\alpha}$ lead to
\begin{equation*}
\sum_{m=0}^\infty e^{-2r(\frac{p}{2})^\alpha \left((2m+1)^\alpha-1\right)}\leq 1+\frac{2^{\alpha}}{4\alpha^2 rp^\alpha}\leq 
1+\frac{1}{2\alpha}\,.
\end{equation*}

(i) The bound for $p\geq 1$ is simpler and follows directly from
\eqref{eq:bound subexp} with the choice $u={2r(\frac{1}{2})^\alpha}$
and \eqref{eq:z bound}. 
\end{proof}
For sub-exponential weights Theorem \ref{tm:WA} takes the following
shape. 
\begin{corollary}[sub-exponential weights]\label{cor:subexp a and b}
Assume that $0<\alpha,\beta\leq 1$, that $r,s>0$, and $0<h\leq 1\leq p$. If
$f\in W(\mathcal{C},\ell^2_{v_{r,\alpha}})$ and $\hat{f}\in
W(\mathcal{C},\ell^2_{v_{s,\beta}})$, then the following error
estimates hold. 
\begin{itemize}
\item[(i)] There are constants $c_{r,\alpha}, c_{s,\beta}>0$ such that 
\begin{equation*}
E^{[n]}_{h}(f)\leq c_{r,\alpha} e^{-r(\frac{p}{2})^\alpha} \|f\|_{W(\mathcal{C},\ell^2_{v_{r,\alpha}})}+c_{s,\beta}e^{-s(2h)^{-\beta}} \|\hat{f}\|_{W(\mathcal{C},\ell^2_{v_{s,\beta}})}\,.
\end{equation*}
If $p\geq \big(\frac{2^\alpha}{2r\alpha}\big)^{1/\alpha}$, then we may
choose $c_{r,\alpha}=e^r(4+\frac{2}{\alpha})^{1/2}$, and if $h\leq
\big(\frac{2s\beta}{2^\beta}\big)^{1/\beta}$, then  $c_{s,\beta}=e^s(4+\frac{2}{\beta})^{1/2}$. 
\item[(ii)] If the step size satisfies $h = n^{-\frac{\alpha}{\alpha+\beta}} \left(\tfrac{s}{r}\right)^{\frac{1}{\alpha+\beta}} 2^{\frac{\alpha-\beta}{\alpha+\beta}}$, then 
\begin{equation*}
E^{[n]}_h(f) \lesssim e^{-s(2h)^{-\beta}} \big(\|f\|_{W(\mathcal{C},\ell^2_{v_{r,\alpha}})}+\|\hat{f}\|_{W(\mathcal{C},\ell^2_{v_{s,\beta}})}\big)\,.
\end{equation*}
\end{itemize}
\end{corollary}
\begin{proof}
The proof is analogous to the one for polynomial weights (Corollary~\ref{cor:poly a and b}). We only check  the constants. 
According to Theorem \ref{tm:WA} and Lemma \ref{lemma:Phi subexp}, the constant $c_{r,\alpha}$ can be chosen as 
\begin{equation*}
c_{r,\alpha} =
v_{r,\alpha}(1)(2+\tfrac{1}{\alpha})^{1/2}\sqrt{1+h}\leq e^r(4+\tfrac{2}{\alpha})^{1/2}\,,
\end{equation*}
where we used $\sqrt{1+h}\leq \sqrt{2}$. Likewise, we get the specification of $c_{s,\beta}$.
\end{proof}

\begin{proof}[Proof of Theorem \ref{tm:intro1}]
  Assume that $\sup_{x\in\R}|f(x)|e^{r|x|^\alpha }<\infty$ and $r'=r-\epsilon
  <r$. Then $f\in W(\mathcal{C}, \ell ^2_{r-\epsilon,\alpha })$ for
  every $\epsilon >0$ with the same argument as in \eqref{later1}. 
Therefore the theorem  is a consequence of Corollary
\ref{cor:subexp a and b}. 
\end{proof}

\subsection{Mixed weights}\label{sec:mw}
 We next  discuss sub-exponential decay of $f$ and  polynomial decay of $\hat f$.
Results for sub-exponential decay of $\hat f$ and  polynomial decay of $f$  are obtained by
switching the roles of $f$ and $\hat f$ in the results below. 

Assume that  $f\in W(\mathcal{C},\ell^2_{v_{r,\alpha}})$ and
$\hat{f}\in W(\mathcal{C},\ell^2_{v_{\beta}})$. We feed  the estimates
of  Lemma~\ref{lemma:bound poly} and \ref{lemma:Phi subexp} into
Theorem~\ref{tm:WA} and  obtain
\begin{equation}\label{eq:bound exp with poly}
E^{[n]}_{h}(f)\leq  
e^r\big(4+\tfrac{2}{\alpha}\big)^{1/2} e^{-r(\frac{p}{2})^\alpha} 
\|f\|_{W(\mathcal{C},\ell^2_{v_{r,\alpha}})}+2^{2\beta +1} \big(1+\tfrac{1}{4\beta-2}\big)^{ 1 /2} h^{\beta} \|\hat{f}\|_{W(\mathcal{C},\ell^2_{v_{\beta}})}
\end{equation}
provided that $p\geq \big(\frac{2^\alpha}{2r\alpha}\big)^{1/\alpha}$. 
Both terms contribute in a balanced manner if
$e^{-r(\frac{p}{2})^\alpha} = h^{\beta}$.  Let  $\mathcal{W}$ be the
Lambert $W$-function, i.e., the inverse of $t\mapsto t e^{t}$. Using
$p=nh$ we express $h$ as 
\begin{equation*}
  h = n^{-1} 2(\tfrac{\beta}{\alpha r})^{1/\alpha}\mathcal{W}^{1/\alpha}(\tfrac{\alpha n^\alpha r}{2^\alpha \beta})\,.
\end{equation*}

As for direct decay conditions, if $f$ and $\hat{f}$ are continuous and satisfy \newline $\sup_{x\in\R}|f(x)|e^{r|x|^\alpha}<\infty$ and $\sup_{\xi\in\R}|\hat{f}(\xi)| (1+|\xi|)^{b}<\infty$ with $b>1$, then, for all $r'<r$ and $\beta<b-\frac{1}{2}$, 
\begin{equation}\label{eq:bound exp with poly11}
E^{[n]}_{h}(f)\lesssim e^{-r'(\frac{p}{2})^{\alpha}} \sup_{x\in\R} |f(x)|e^{r|x|} + h^{\beta}\sup_{\xi\in\R} |\hat{f}(\xi)|(1+|\xi|)^b\,.
\end{equation} 
This yields Theorem \ref{thm:mixed intro} in the introduction.

For the exponential case $\alpha=1$, a related  bound for the sup-norm
of $ \hat{f}_{\frac{1}{p},n} - \DFT\! f_{h,n}$ is derived by Briggs
and Henson~\cite[Theorem 6.6]{Briggs}.  However, they   impose
 restrictive boundary conditions.  For a comparable bound
the periodization $\mathcal{P}_p(f{\bf
  1}_{[-\frac{p}{2},\frac{p}{2})]})$ must be $b-2$-times continuously
differentiable on $\R$. These are already violated when $f(p/2) \neq
f(-p/2)$, in which case  \eqref{eq:bound exp with poly11} holds only
with $\beta =1$ regardless of the global smoothness of $f$.

\section{Interpolation}\label{sec:bandlimited}
So far, we have approximated the Fourier samples $\hat{f}(\tfrac{k}{p})$ by the discrete Fourier transform $h\sum_{j\in[n]} f(hj)
e^{-2\pi\mathrm{i} kj /n }= \sqrt{p}\,(\mathcal{F}f_{h,n})(k)$, for $k\in[n]$. 
In this section we change the point of view: we now  interpolate the vector
$\sqrt{p}\,\mathcal{F}f_{h,n}$ and study how the interpolating function
approximates $\hat{f}$  on whole real line 
$\R $.   To do so, we use 
three exemplary interpolation schemes that are often used in
approximation theory, namely piecewise constant and piecewise linear
interpolation, and interpolation using sinc functions.

The standard procedure starts with a cardinal interpolating function:
this is a nice function $\phi $ on $\R$ that satisfies the interpolation
condition $\phi (k) = \delta _{k,0}$. The approximation of $\hat{f}$ on
$\R $ from the discrete Fourier transform is then given by
\begin{equation}
  \label{eq:approx f hat on R}
\Psi _{h,n} (\xi ) = \sqrt p \, \sum _{k\in [n]} \mathcal{F}f_{h,n}(k) \phi
(p\xi -k)
=    h \sum_{k,j\in[n]} f(hj) e^{-2\pi\mathrm{i}\frac{kj}{n}}\phi
(p\xi-k) \, . 
\end{equation}
This expression is an approximation of $\hat f $ on $\R $ that
satisfies $\Psi _{h,n}(\tfrac{k}{p}) = \sqrt p \,
\mathcal{F}f_{h,n}(k)$, and we seek to estimate the global error
$\|\hat f  - \Psi _{h,n}\|_{L^2}$.

As the cardinal interpolating function $\phi $ we will use 
the first two B-splines $B_1={\bf 1}_{[-\frac{1}{2},\frac{1}{2})}$ and
$B_2(x)=(1-|x|){\bf 1}_{[-1,1]}(x)$ and the cardinal sine function
$\sinc (x) = \tfrac{\sin \pi x}{\pi x}$. Clearly, these functions
satisfy $\phi (k) = \delta _{k,0}$. To keep the  notion compact, we set
$B_0:=\sinc$.
For $\phi = B_1$, the right-hand side of \eqref{eq:approx f hat on R} is a step
function, for $\phi = B_2$ we obtain a continuous, piecewise linear
function (this is how many software programs would plot the approximating vector $\mathcal{F}f_{h,n}$), 
and for $\phi = \sinc$ we obtain a
smooth function whose Fourier transform has support in $[-1/2,1/2]$,
which is commonly referred to as a bandlimited function in time-frequency analysis.

 We first derive an error bound for the bandlimited approximation
\eqref{eq:approx f hat on R} with $\phi = \sinc$ for general
weights. This is then specialized to the polynomial and the
sub-exponential weights, resulting in explicit error rates. The
approximation rates for the $B$-splines \eqref{eq:approx f hat on R}
with $i=1,2$ are restricted to decay rates  $\alpha\leq 1$ and $\alpha\leq 2$, respectively.
\begin{theorem}\label{tm:finale again11}
  Assume that $0<h\leq 1\leq p$ and that the weights  $v,w$ satisfy
  \eqref{eq: weights}.  If $f\in W(\mathcal{C},\ell^2_{v})$ and $\hat{f}\in W(\mathcal{C},\ell^2_{w})$,  then 
\begin{equation} \label{hmhm3}
\hspace{-1ex}\Big\|\hat{f} - h \!\!\sum_{k,j\in[n]} \!\!f(hj) e^{-2\pi\mathrm{i}\frac{kj}{n}}\sinc(p \xi-k) \Big\|_{L^2} 
\lesssim  \Phi_v(p)\|f\|_{W(\mathcal{C},\ell^2_{v})}+ 
 \Phi_w(h^{-1})\|\hat{f}\|_{W(\mathcal{C},\ell^2_{w})}\,.
\end{equation}
If $\sinc$ is replaced with $B_i$, for $i=1,2$, then the estimate
\eqref{hmhm3} still holds for the polynomial weight  $v=v_\alpha$
subject to  the restriction $1/2<\alpha\leq i$. 
\end{theorem}
We note that the error is of the same form as in
Theorem~\ref{tm:WA}. In particular,
(i) for the  polynomial weights  $v=v_\alpha$ and $w=v_\beta$ with $\alpha,\beta>\frac 1 2$ we obtain $\Phi_v(p) \lesssim p^{-\alpha}$ and $\Phi_w(h\inv)\lesssim h^{\beta}$ (see Lemma \ref{lemma:bound poly}), 
 and the choice $h\asymp n^{-\frac{\alpha}{\alpha+\beta}}$ leads to 
\begin{equation*}
\Big\|\hat{f} - h \!\!\sum_{k,j\in[n]} \!\!f(hj) e^{-2\pi\mathrm{i}\frac{kj}{n}}\sinc(p \xi-k) \Big\|_{L^2} \lesssim  n^{-\frac{\alpha \beta}{\alpha+\beta}} 
\left(\|f\|_{W(\mathcal{C},\ell^2_{v_\alpha})}+ 
 \|\hat{f}\|_{W(\mathcal{C},\ell^2_{v_\beta})}\right)\,.
\end{equation*}
(ii) For sub-exponential weights $v=v_{r,\alpha}$ and $w=v_{s,\beta}$ with
$0<\alpha,\beta\leq1$ and $r,s>0$ we obtain $\Phi_v(p) \lesssim e^{-r
  (\frac{p}{2})^{\alpha}}$ and $\Phi_w(h\inv)\lesssim
e^{-s(\frac{1}{2h})^{\beta}}$ (see Lemma \ref{lemma:Phi subexp}).

(iii) For exponential weights $v(x) = e^{r|x|}, w(x) = e^{s|x|}$ and 
the choice $p=\sqrt{n}, h=1/\sqrt{n}$ we obtain the root-exponential
convergence 
$$
\Big\|\hat{f} - h \!\!\sum_{k,j\in[n]} \!\!f(hj)
e^{-2\pi\mathrm{i}\frac{kj}{n}}\sinc(p \xi-k) \Big\|_{L^2} \lesssim
 e^{-d \sqrt{n}} \, ,
 $$
 as in ~\cite{StengerBook}. 
\begin{proof}[Proof of Theorem~\ref{tm:finale again11}]
The proof is based on the following three approximation steps:
\begin{align}
\hat{f}(\xi) \, \overset{(s1)}{\approx} \,
\sum_{k\in\Z} \hat{f}(\tfrac{k}{p}) B_i(p\xi-k)
&\, \overset{(s2)}{\approx} \,
 \sum_{k\in[n]} \hat{f}(\tfrac{k}{p}) B_i(p \xi-k) \notag \\
& \, \overset{(s3)}{\approx} \,
  h \sum_{k,j\in[n]} f(hj)  e^{-2\pi\mathrm{i}\frac{kj}{n}}B_i(p \xi-k)\,. \label{hmhm4}
\end{align}
Here (s1) is the interpolation error, (s2) the truncation error, and
(s3) the discretization error.
We derive  bounds for the error in the steps (s2) and (s3) with respect to general
weights simultaneously for $i=0,1,2$. 

(s3) For $i=0,1$, the system $\{\sqrt{p} B_i(p\xi-k)\}_{k\in\Z}$ is
orthonormal in $L^2(\R)$ and it is still a Riesz sequence for $i=2$
with Riesz bounds independent of $p$ (this can be verified by direct
computation). Therefore, we conclude that 
\begin{align*}
\lefteqn{\Big\|\sum_{k\in[n]} \hat{f}(\tfrac{k}{p}) B_i(p \xi-k) - h
  \sum_{k,j\in[n]} f(hj)
  e^{-2\pi\mathrm{i}\frac{kj}{n}}B_i(p \xi-k)\Big\|_{L^2}} \\
 & = \Big\|  \sum _{k\in [n]} \Big(\tfrac{1}{\sqrt p} \hat{f}(\tfrac{k}{p}) -
  \mathcal{F}f_{h,n}(k)\Big) \, \sqrt p \, B_i(p \xi - k) \Big\|_{L^2}  \asymp \big\|\hat{f}_{\frac{1}{p},n} - \mathcal{F} f_{h,n}\big\|\ = E^{[n]}_{h}(f) \,,
\end{align*}
and this is just the discrete approximation error that has been
treated  in Theorem \ref{tm:WA}.

(s2) Using orthonormality or the Riesz basis property again,  we observe 
\begin{equation*}
\Big\|\sum_{k\in\Z} \hat{f}(\tfrac{k}{p}) B_i(p \xi-k) - \sum_{k\in[n]} \hat{f}(\tfrac{k}{p}) B_i(p \xi-k) \Big\|^2_{L^2} \asymp  p^{-1} \sum_{k\in\Z\setminus[n]} |\hat{f}(\tfrac{k}{p})|^2\,.
\end{equation*}
 The monotonicity \eqref{eq:b} of the weight $w$  implies 
\begin{equation*}
p^{-1}\sum_{k\in\Z\setminus[n]} |\hat{f}(\tfrac{k}{p})|^2 
 \lesssim p^{-1}|w(\tfrac{n}{2p})|^{-2} \sum_{k\in\Z\setminus[n]} |\hat{f}(\tfrac{k}{p})|^2|w(\tfrac{k}{p})|^{2}\,.
\end{equation*}
The sum on the right-hand side  is bounded as in the estimates at the end of the proof of Lemma \ref{lemma:op norm sampling}. We deduce 
\begin{equation*}
p^{-1}\sum_{k\in\Z\setminus[n]} |\hat{f}(\tfrac{k}{p})|^2|w(\tfrac{k}{p})|^{2} \leq p^{-1}\lceil p \rceil |w(1)|^2 \|\hat{f}\|^2_{W(\mathcal{C},\ell^2_{w})}\,.
\end{equation*}
As $|w(\tfrac{n}{2p})|^{-1}=|w(\tfrac{1}{2h})|^{-1} \leq \big(2\sum_{m=0}^\infty |w(\frac{m}{h}+\frac{1}{2h})|^{-2}\big)^{1/2}=\Phi_w(h^{-1})$ we obtain
\begin{equation*}
\Big\|\sum_{k\in\Z} \hat{f}(\tfrac{k}{p}) B_i(p \xi-k) - \sum_{k\in[n]} \hat{f}(\tfrac{k}{p}) B_i(p \xi-k) \Big\|_{L^2} \lesssim \Phi_w(h^{-1})\|\hat{f}\|_{W(\mathcal{C},\ell^2_{w})}\,.
\end{equation*}

(s1) The interpolation error $\|\hat{f}-\sum_{k\in\Z} \hat{f}(\tfrac{k}{p})
B_i(p\xi-k)\|_{L^2}$ relates to classical interpolation of $\hat{f}$
from its samples \cite{butzer-2,butzer1,StengerBook}. We carry out the calculations for $B_0=\sinc$. 
Direct computations and 
Plancherel's formula lead to 
\begin{align*}
\Big\|\hat{f} -  \sum_{k\in\Z} \hat{f}(\tfrac{k}{p}) \sinc(p \xi -k) \Big\|^2_{L^2} & = \Big\|f - \frac{1}{p}\sum_{k\in\Z} \hat{f}(\tfrac{k}{p}) e^{2\pi\mathrm{i} x\frac{k}{p}}{\bf 1}_{[-\frac{p}{2},\frac{p}{2}]}\Big\|^2_{L^2} \\
& = \int_{|x|>\frac{p}{2}} \!|f(x)|^2 \mathrm{d}x + \!
\int_{-\frac{p}{2}}^{\frac{p}{2}} \!\bigg| f(x)- \frac{1}{p}\sum_{k\in\Z} \hat{f}(\tfrac{k}{p}) e^{2\pi\mathrm{i} x \frac{k}{p}}\bigg|^2\mathrm{d}x\,.
\end{align*}
The first term can be estimated as 
\begin{equation*}
\int_{|x|>\frac{p}{2}} \left| f(x)\right|^2 \mathrm{d}x 
\leq 
 |v(\tfrac{p}{2})|^{-2} \|f\|^2_{L^2_{v}}\lesssim \Phi^2_v(p)\|f\|_{W(\mathcal{C},\ell^2_v)}\,,
\end{equation*}
where we used the continuous embedding $W(\mathcal{C},\ell^2_{v})\subseteq L^2_{v}$ and $|v(\tfrac{p}{2})|^{-2}\leq \Phi^2_v(p)$.

For the second term  
we use the Poisson summation formula  
and  apply  Lemma \ref{lemma:op norm sampling}(i), so that  
\begin{align*}
\int_{-\frac{p}{2}}^{\frac{p}{2}} \Big| f(x)- p^{-1}\sum_{k\in\Z}
  \hat{f}(\tfrac{k}{p}) e^{2\pi\mathrm{i} x
  \frac{k}{p}}\Big|^2\mathrm{d}x & = \|f-\cPp f\|^2_{L^2(-\frac{p}{2},\frac{p}{2})}
  \lesssim  \Phi^2_{v}(p)\|f\|^2_{W(\mathcal{C},\ell^2_v)}\, .
\end{align*}
This  provides the  bound for the interpolation error as  
\begin{equation*}
\Big\|\hat{f}-\sum_{k\in\Z} \hat{f}(\tfrac{k}{p}) \sinc(p\xi-k)\Big\|_{L^2}\lesssim \Phi_{v}(p)\|f\|_{W(\mathcal{C},\ell^2_v)}\,.
\end{equation*}

To treat the cases $i=1,2$ of piecewise constant and piecewise linear
interpolation for the polynomial weight $v=v_\alpha$, 
we use  results from \cite[Theorem 3 and Lemma 3]{MR1367171}. For
$\alpha\leq i, i=1,2,$ these  imply
\begin{equation}\label{eq:brambles bound}
\Big\|\hat{f}-\sum_{k\in\Z} \hat{f}(\tfrac{k}{p})
B_i(p\xi-k)\Big\|_{L^2} \lesssim p^{-\alpha}
\|f\|_{L^2_{v_\alpha}} \, . 
\end{equation}

The continuous embedding $W(\mathcal{C},\ell^2_{v_\alpha})\subseteq
L^2_{v_\alpha}$ yields 
\begin{equation*}
\|\hat{f}-\sum_{k\in\Z} \hat{f}(\tfrac{k}{p})
B_i(p\xi-k)\|_{L^2} \lesssim p^{-\alpha}
\|f\|_{W(\mathcal{C},\ell^2_{v_\alpha})}\,.
\end{equation*} 
The final error is the sum of the three errors (s1), (s2), and (s3). 
\end{proof}

\begin{remark}\label{rem:order}
The restriction to $\alpha\leq i$, for $i=1,2$, in the second part of Theorem
\ref{tm:finale again11} stems from the bound \eqref{eq:brambles
  bound}. For stronger polynomial decay $\alpha$, the splines $B_1$
and $B_2$ need to be replaced by a suitable  function $\phi_\alpha$, see
\cite{butzer-2} for examples. We need at least that  $\phi_\alpha$
interpolates on $\Z$ and that 
 its integer shifts  form a Riesz-sequence in $L^2$.
\end{remark}
The proof for the interpolating function $\phi = \sinc$ holds for all
polynomial weights $v_\alpha$ with $\alpha > \frac{1}{2}$. The
$\sinc$-function arises naturally in Fourier approximation, but its
slow decay poses challenges in numerical analysis. Nevertheless, it
has been successfully employed in various numerical methods -- see
Stenger's work~\cite{Stenger81,StengerBook} for comprehensive
treatments of $\sinc$-based techniques.

\section{Minimal requirements for convergence}\label{sec:minimal}
In this section we aim to identify the weakest conditions under which the Fourier transform can be successfully approximated
by the discrete Fourier transform.   Although this question may not be of immediate 
 practical value, it is a fundamental question of intrinsic
 mathematical interest.
 
To identify weak conditions, we now  choose the weaker norm 
\begin{equation}\label{eq:sup now}
\sup_{k\in[n]}\Big|\hat{f}(\tfrac{k}{p}) - h\sum_{j\in[n]} f(hj)e^{-2\pi\mathrm{i} \frac{k j}{n}} \Big|\,
\end{equation}
and ask for sufficient conditions on $f$ and $\hat{f}$ such that
\eqref{eq:sup now} converges to $0$ when $h\rightarrow 0$ and
$p\rightarrow \infty$.  
It turns out that the same general conditions for the validity of the
  Poisson summation formula   also guarantee the
  convergence. 

\begin{proposition}\label{prop:instruction}
If $f,\hat{f}\in W(\mathcal{C},\ell^1)$, then 
\begin{equation*}
\lim_{\substack{h\rightarrow 0\\ p\rightarrow\infty}} \,\sup_{k\in[n]} \Big| \hat{f}(\tfrac{k}{p}) - h\sum_{j\in[n]} f(hj)e^{-2\pi\mathrm{i} \frac{k j}{n}}\Big| =0\,.
\end{equation*}
\end{proposition}
The limit above is taken in the sense that $h\to 0$ and $p^{-1} \to
0$ in arbitrary relation to each other. In particular, it implies that $n=\frac{p}{h}\rightarrow\infty$.
\begin{proof}
Following the general approach outlined in Section \ref{sec:decomp}, we 
decompose the error $\sup_{k\in[n]}\Big|\hat{f}(\tfrac{k}{p}) -
h\sum_{j\in[n]} f(hj)e^{-2\pi\mathrm{i} \frac{k j}{n}} \Big| =
\sqrt{p} \,\|\hat{f}_{\frac{1}{p},n} - \mathcal{F} f_{h,n}\|_\infty$
into a time and a frequency component.  Since we are taking limits, we
may assume without loss of generality that $h\leq 1$.

As in (4.6),    
we have 
\begin{equation}\label{eq:tf comps}
\|\hat{f}_{\frac{1}{p},n} - \mathcal{F} f_{h,n}\|_\infty \leq \|\hat{f}_{\frac{1}{p},n}- (\cPh \hat{f})_{\frac{1}{p},n} \|_\infty + \|\mathcal{F} (\cPp f)_{h,n} - \DFT\! f_{h,n}\|_\infty\,.
\end{equation}
Since  the discrete Fourier transform satisfies
$\|\mathcal{F}y\|_\infty \leq \frac{1}{\sqrt{n}} \|y\|_1$ for
$y\in\C^n$, 
we obtain
\begin{equation}\label{eq:tf comp}
\sup_{k\in[n]}\Big|\hat{f}(\tfrac{k}{p}) -
h\sum_{j\in[n]} f(hj)e^{-2\pi\mathrm{i} \frac{k j}{n}} \Big|
\leq \sqrt{p}\|(\hat{f} - \cPh \hat{f})_{\frac{1}{p},n} \|_\infty +
\sqrt{h}\| (f-\cPp f)_{h,n}\|_1\, ,
\end{equation}
which   splits the error into a time component  and a frequency
component as before.

It remains to obtain suitable decay of samples of $\hat{f} - \cPh
\hat{f}$ and $f-\cPp f$. We start with the frequency component
$\sqrt{p}\|(\hat{f} - \cPh \hat{f})_{\frac{1}{p},n} \|_\infty$. The
observation $|\frac{k}{p}|\leq \frac{1}{2h}$ for $k\in[n]$ and  the
assumption  $h\leq 1$ lead to 
\begin{align} \label{hmhm5}
\sqrt{p}\|(\hat{f} - \cPh \hat{f})_{\frac{1}{p},n}\|_\infty & =  \sup_{k\in[n]} \Big|\sum_{l\in\Z\setminus\{0\}} \hat{f}(\tfrac{k}{p}+h^{-1}l)  \Big| \leq \sum_{|l|\geq \frac{1}{2h}} \sup_{\xi\in[0,1]} |\hat{f}(\xi+l)|\,.
\end{align}
Since $\hat{f}\in W(\mathcal{C},\ell^1)$, this sum tends to zero for $h\rightarrow 0$. 

To estimate the time component $\sqrt{h}\| (f-\cPp f)_{h,n}\|_1$, we compute
\begin{equation*}
\sqrt{h} \| (f-\cPp f)_{h,n}\|_1  = h \sum_{j\in[n]} \Big| \sum_{l\in\Z\setminus\{0\}} f(hj+pl)\Big| \leq h \sum_{j\in\Z\setminus[n]} |f(hj)|\,.
\end{equation*}
Since $h j\in [l,l+1)$ occurs for at most $\lceil
h^{-1}\rceil$ many integers $j$, we obtain  
\begin{equation*}
 \sqrt{h} \| (f-\cPp f)_{h,n}\|_1\leq h \lceil h^{-1}\rceil \sum_{|l|\geq \frac{p}{2}} \sup_{x\in[0,1]} |f(x+l)|\,.
 \end{equation*}
Since $f\in W(\mathcal{C},\ell^1)$, the series tends to zero for
$p\rightarrow \infty$. The factor $h \lceil h^{-1}\rceil \leq
(1+h)\leq 2$ does not cause any issues because  $h\rightarrow 0$. 
\end{proof}

In previous sections, we considered the $\ell^2$-error $E^{[n]}_h(f)$, which is stronger than \eqref{eq:sup now}. 
Since $\|y\|_2 \leq \sqrt n \, \|y\|_\infty $ for $y\in \C ^n$,
Proposition~\ref{prop:instruction} also implies that  for $f,\hat{f}\in W(\mathcal{C},\ell^1)$,
$$
\lim _{\substack{h\rightarrow 0\\ p\rightarrow\infty}}\sqrt h \,
E^{[n]}_h(f)  \leq  \lim_{\substack{h\rightarrow 0\\ p\rightarrow\infty}}
\sup_{k\in[n]} \Big| \hat{f}(\tfrac{k}{p}) - h\sum_{j\in[n]}
f(hj)e^{-2\pi\mathrm{i} \frac{k j}{n}}\Big| \rightarrow 0 \, ,
$$
so that  an additional factor $\sqrt{h}$ is needed for convergence. 

Parallel to Section~\ref{sec:bandlimited} we now investigate when the
interpolation error of $\hat{f}$  in
\eqref{eq:approx f hat on R} 
converges to zero.     
\begin{proposition}\label{prop:Kaiblinger}
If $f,\hat{f}\in W(\mathcal{C},\ell^1)$, then for interpolation with
step functions $B_1$ and with piecewise linear functions $B_2$ we
obtain 
\begin{equation}\label{eq:2}
\lim_{\substack{h\rightarrow 0\\ p\rightarrow\infty}} \,\sup_{\xi\in\R}\Big| \hat{f}(\xi) - h \sum_{k,j\in[n]} f(hj) e^{-2\pi\mathrm{i}\frac{kj}{n}} B_i(p\xi-k) \Big | = 0\,,\qquad i=1,2.
\end{equation}
\end{proposition}
\begin{proof}
As in the proof of Theorem~\ref{tm:finale again11} we split the
approximation error  into interpolation, truncation, and
 discretization via the discrete Fourier transform as follows: 
\begin{align*}
\hat{f}(\xi) \, \overset{(a1)}{\approx} \, \sum_{k\in\Z}
   \hat{f}(\tfrac{k}{p}) B_i(p\xi-k)&\, \overset{(a2)}{\approx} \,  \sum_{k\in[n]}
  \hat{f}(\tfrac{k}{p})  B_i(p\xi-k) \\ &\, \overset{(a3)}{\approx} \,h
    \sum_{k,j\in[n]} f(hj) e^{-2\pi\mathrm{i}\frac{kj}{n}} B_i(p\xi-k)\, .
  \end{align*}
To estimate the interpolation error  $(a1)$, we use the fact that
$B_i$ has compact support, $\supp(B_i)\subseteq[-1,1]$, and that  the integer
shifts of $B_i$, $i=1,2$, form  a partition of unity. This leads to 
\begin{align*}
\Big| \hat{f}(\xi)- \sum_{k\in\Z} \hat{f}(\tfrac{k}{p}) \,  B_i(p\xi
  -k)\Big|  & =\Big|\sum_{k\in\Z}
              (\hat{f}(\xi)-\hat{f}(\tfrac{k}{p}))\, B_i(p\xi-k)\Big|\\
& \leq  \sum_{\substack{k\in\Z \\|\xi-\frac{k}{p}|\leq \frac{1}{p}}}
  |\hat{f}(\xi)-\hat{f}(\tfrac{k}{p})|\, B_i(p\xi-k)\\
& \leq \sup _{k\in \Z} \sup _{\xi \in \R: |\xi - \frac{k}{p}|\leq \frac{1}{p}}|\hat{f}(\xi)-\hat{f}(\tfrac{k}{p})| \sum  _{k\in\Z} \, B_i(p\xi-k)\\
  &\leq \sup _{\xi  , \eta \in \R: |\xi - \eta | \leq \frac{1}{p}} |\hat
    f(\xi ) - \hat{f}(\eta )| \,.                                                                         
\end{align*}
Since $f\in W(\mathcal{C}, \ell ^1)$ is uniformly continuous and
$p\rightarrow \infty$, we conclude that  
\begin{equation*}
\lim_{p\rightarrow\infty} \sup_{\xi\in\R} \Big|\hat{f}(\xi)-\sum_{k\in\Z}\hat{f}(\tfrac{k}{p})\, B_i(p\xi-k)\Big|=0\,.
\end{equation*}

For the truncation error  $(a2)$  we use $\|B_i\|_\infty =1$ and find  for all $\xi \in \R $ that 
\begin{align*}
\Big|\sum_{k\in\Z}\hat{f}(\tfrac{k}{p}) \, B_i(p\xi-k)-\sum_{k\in[n]}\hat{f}(\tfrac{k}{p}) \, B_i(p\xi-k) \Big| &=  \Big|\sum_{k\in\Z\setminus[n]}\hat{f}(\tfrac{k}{p}) \, B_i(p\xi-k) \Big| \\
&\leq \|B_i\|_\infty  \sum_{\substack{k\in\Z\setminus[n]\\ |\xi-\frac{k}{p}|\leq \frac{1}{p}}}\big|\hat{f}(\tfrac{k}{p}) \big| \\
&
\leq  \sum_{k\in\Z\setminus[n]}\big|\hat{f}(\tfrac{k}{p}) \big| \, .
\end{align*}
Since $p\geq 1$ and 
$\frac{|k|}{p} \geq \frac{n}{2} \, \frac{1}{nh} = \frac{1}{2h}$ for
$k\in [n]$, the last sum is estimated  as in \eqref{hmhm5} by
$$
\sum_{k\in\Z\setminus[n]}\big|\hat{f}(\tfrac{k}{p}) \big| \leq \sum
_{|l|\geq \frac{1}{2h}} \sup _{\xi \in [0,1]} |\hat f (\xi +l)| \to 0
\, .
$$
Since $\hat f\in W(\mathcal{C},\ell ^1)$, we obtain
\begin{equation*}
\lim_{h\rightarrow 0}\sup_{\xi\in\R}\Big|\sum_{k\in\Z}\hat{f}(\tfrac{k}{p}) \, B_i(p\xi-k)-\sum_{k\in[n]}\hat{f}(\tfrac{k}{p}) \, B_i(p\xi-k) \Big| =0\,.
\end{equation*}

For the discretization error  $(a3)$, we estimate 
\begin{align*}
& \hspace{-2cm}\bigg|\sum_{k\in[n]} \Big(\hat{f}(\tfrac{k}{p}) - h \sum_{j\in[n]} f(hj) e^{-2\pi\mathrm{i}\frac{kj}{n}}\Big) \, B_i(p\xi-k) \bigg| \\
&\qquad \leq \sup_{k\in[n]} \Big|\hat{f}(\tfrac{k}{p}) - h \sum_{j\in[n]} f(hj) e^{-2\pi\mathrm{i}\frac{kj}{n}} \Big|\,  \Big|\sum_{k\in[n]}  B_i(p\xi-k) \Big| \\
&\qquad \leq \sup_{k\in[n]} \Big|\hat{f}(\tfrac{k}{p}) - h \sum_{j\in[n]} f(hj) e^{-2\pi\mathrm{i}\frac{kj}{n}} \Big|\, ,
\end{align*}
where we have used again that the integer translates of the cardinal $B$-splines form a partition of unity. According to Proposition \ref{prop:instruction} this term vanishes when $p\rightarrow \infty$ and $h\rightarrow 0$.

The final error is the sum of the three errors (a1), (a2), and (a3). 
\end{proof}

\begin{remark}
(i) The formulation of  Proposition~\ref{prop:c6} in the introduction
  follows, because the condition $\sup _{x\in \R } |f(x)|
  (1+|x|)^{1+\epsilon } <\infty$ implies that $f\in W(\C, \ell ^1)$.

  (ii) Convergence in a stronger norm for the choice $p=1/h = \sqrt n$ and under stronger conditions on $f$ and $\hat f$ was 
derived in \cite{Kaiblinger05} for the piecewise linear interpolation
with $B_2$.  
Proposition~\ref{prop:Kaiblinger} extends the applicability of the main result in \cite{Kaiblinger05} to a larger class of functions. 
For the rather
non-trivial  comparison
of the condition $f,\hat f \in W(\mathcal{C}, \ell ^1)$ to the one in
\cite{Kaiblinger05} we refer to  \cite[Theorem
2]{Losert80}. 

(iii) For $B_0=\sinc $ the above proof of Proposition 
\ref{prop:Kaiblinger} breaks down at several steps, because the
integer shifts of $\sinc $  are not summable. We do not know whether Proposition
\ref{prop:Kaiblinger} remains true in this case. 
\end{remark}

\end{document}